\documentclass[a4paper,british]{amsart}
\usepackage{amscd}
\usepackage{amssymb}
\usepackage{enumitem}

\usepackage[applemac]{inputenc}
\usepackage{amsmath, amsthm, amssymb}
\usepackage[french,english]{babel}
\usepackage{hyperref}

\usepackage{mathtools}
\usepackage[arrow,curve,matrix]{xy}

\usepackage{colortbl}
\usepackage{graphicx}
\usepackage{tikz}
\usepackage{tikz-cd}

\usepackage{comment}

\usepackage{mathrsfs}
\usepackage[retainorgcmds]{IEEEtrantools}

\DeclareFontFamily{OMS}{rsfs}{\skewchar\font'60}
\DeclareFontShape{OMS}{rsfs}{m}{n}{<-5>rsfs5 <5-7>rsfs7 <7->rsfs10 }{}
\DeclareSymbolFont{rsfs}{OMS}{rsfs}{m}{n}
\DeclareSymbolFontAlphabet{\scr}{rsfs}

\definecolor{linkred}{rgb}{0.7,0.2,0.2}
\definecolor{linkblue}{rgb}{0,0.2,0.6}

\setdescription{labelindent=\parindent, leftmargin=2\parindent}
\setitemize[1]{labelindent=\parindent, leftmargin=2\parindent}
\setenumerate[1]{labelindent=0cm, leftmargin=*, widest=iiii}

\numberwithin{figure}{section}

\usepackage[hyperpageref]{backref}

\usepackage{mathpazo}

\def\bB{{\mathbb B}}

\def\bR{{\mathbb R}}
\def\bN{{\mathbb N}}
\def\bQ{{\mathbb Q}}

\newcommand{\sA}{\scr{A}}
\newcommand{\sB}{\scr{B}}
\newcommand{\sC}{\scr{C}}
\newcommand{\sD}{\scr{D}}
\newcommand{\sE}{\scr{E}}
\newcommand{\sF}{\scr{F}}
\newcommand{\sG}{\scr{G}}
\newcommand{\sH}{\scr{H}}

\newcommand{\sK}{\scr{K}}
\newcommand{\sL}{\scr{L}}
\newcommand{\sM}{\scr{M}}
\newcommand{\sN}{\scr{N}}
\newcommand{\sO}{\scr{O}}

\newcommand{\sQ}{\scr{Q}}

\newcommand{\sT}{\scr{T}}

\newcommand{\wtilde}{\widetilde}

\newcommand{\onto}{\twoheadrightarrow}

\newcommand{\hooklongrightarrow}{\lhook\joinrel\longrightarrow}

\theoremstyle{plain}
\newtheorem{theorem}{Theorem}[section]

\newtheorem{proposition}[theorem]{Proposition}
\newtheorem{corollary}[theorem]{Corollary}
\newtheorem{lemma}[theorem]{Lemma}
\newtheorem{conjecture}[theorem]{Conjecture}

\theoremstyle{definition}
\newtheorem{definition}[theorem]{Definition}

\theoremstyle{question}
\theoremstyle{remark}
\newtheorem{remark}[theorem]{Remark}
\newtheorem{notation}[theorem]{Notation}

\newtheorem{set-up}[theorem]{Set-up}
\newtheorem{claim}[theorem]{Claim}

\setlist[enumerate]{label=(\thetheorem.\arabic*), before={\setcounter{enumi}{\value{equation}}}, after={\setcounter{equation}{\value{enumi}}}}

\numberwithin{equation}{theorem}

\newcommand{\newabstract}[1]{%
  \par\bigskip
  \csname otherlanguage*\endcsname{#1}%
  \csname captions#1\endcsname
  \item[\hskip\labelsep\scshape\abstractname.]
}
\DeclareMathOperator{\Amp}{Amp}
\DeclareMathOperator{\Eff}{Eff}
\DeclareMathOperator{\Bs}{Bs}

\DeclareMathOperator{\codim}{codim}

\DeclareMathOperator{\WDiv}{WDiv}

\DeclareMathOperator{\Gal}{Gal}
\DeclareMathOperator{\HN}{HN}

\DeclareMathOperator{\Mov}{Mov}
\DeclareMathOperator{\N}{N}
\DeclareMathOperator{\rank}{rank}
\DeclareMathOperator{\reg}{reg}

\DeclareMathOperator{\supp}{supp}

\makeatletter
\let\saveqed\qed
\renewcommand\qed{%
   \ifmmode\displaymath@qed
   \else\saveqed
   \fi}
\makeatother

\makeatletter
\hypersetup{
  pdfauthor={\authors},
  pdftitle={\@title},
  pdfsubject={\@subjclass},
  pdfkeywords={\@keywords},
  pdfstartview={Fit},
  pdfpagelayout={TwoColumnRight},
  pdfpagemode={UseOutlines},
  bookmarks,
  colorlinks,
  linkcolor=linkblue,
  citecolor=linkred,
  urlcolor=linkred}
\makeatother

\begin{document}
\bibliographystyle{alpha}



\title[Orbifold Chern classes inequalities]{Orbifold Chern classes inequalities and applications}

\author{Erwan Rousseau}
\address{Erwan Rousseau \\ Institut Universitaire de France
	\& Univ Brest\\ UMR CNRS 6205\\\ Laboratoire de Math\'ematiques de Bretagne Atlantique\\ 
		6, Avenue Victor Le Gorgeu, 29200 Brest\\
		France} 
\email{\href{mailto:erwan.rousseau@univ-brest.fr}{erwan.rousseau@univ-brest.fr}}
\urladdr{\href{http://eroussea.perso.math.cnrs.fr/}{http://eroussea.perso.math.cnrs.fr/}}

\author{Behrouz Taji}

\address{Behrouz Taji, The University of Sydney, School of Mathematics and Statistics F07, NSW 2006 Australia}
\email{\href{mailto:behrouz.taji@sydney.edu.au}{behrouz.taji@sydney.edu.au}}
\urladdr{\href{http://www.maths.usyd.edu.au/u/behrouzt/}{http://www.maths.usyd.edu.au/u/behrouzt/}}

\thanks{Behrouz Taji was partially supported by the DFG-Graduiertenkolleg GK1821
  ``Cohomological Methods in Geometry'' at the University of Freiburg. Erwan Rousseau was partially supported by the ANR project \lq\lq FOLIAGE\rq\rq{}, ANR-16-CE40-0008.}

\begin{abstract}
In this paper we prove that given a pair $(X,D)$ of a threefold $X$ and a boundary divisor $D$
with mild singularities, if $(K_X+D)$ is 
movable, then the orbifold second Chern class $c_2$ of $(X,D)$ is pseudoeffective. 
This generalizes the classical result of Miyaoka on the pseudoeffectivity of $c_2$
for minimal models. As an application, we give a simple solution to Kawamata's effective 
non-vanishing conjecture in dimension $3$, where we prove that $H^0(X, K_X+H)\neq 0$,
whenever $K_X+H$ is nef and $H$ is an ample, effective, reduced Cartier divisor.
Furthermore, we study Lang--Vojta's 
conjecture for codimension one subvarieties and prove that minimal threefolds of general type
have only finitely many Fano, Calabi-Yau or Abelian 
subvarieties of codimension one that are mildly singular and whose numerical classes 
belong to the movable cone. 

\end{abstract}

\maketitle

\tableofcontents

\section{Introduction}
It is well known that the Chern classes of nef vector bundles over smooth projective 
varieties satisfy certain inequalities \cite{DPS}. 
More generally, 
a theorem of Miyaoka \cite{Miyaoka87} states that over a normal, projective variety (that is smooth in codimension two)
any torsion free, coherent sheaf $\sE$ that is \emph{semipositive} with respect to the tuple of ample 
divisors $(H_1,\dots,H_{n-1})$ and whose determinant $\det(\sE)$ is nef, 
verifies the inequality 

$$c_2(\sE)\cdot H_1\dots H_{n-2} \geq 0.$$

On the other hand, thanks to Miyaoka's celebrated \emph{generic semipositivity} result, cf.~\cite{Miyaoka87}, and  
the result of Boucksom, Demailly, P\u{a}un and Peternell (\cite{BDPP}), when $K_X$ is pseudoeffective, 
the cotangent bundle $\Omega^1_X$ of a smooth projective variety 
is generically semipositive.
As a result, for a smooth projective variety $X$ with $K_X$ nef, the inequality
\begin{equation}\label{eq:MiyChern}
c_2(X)\cdot H_1\dots H_{n-2} \geq 0
\end{equation}
holds, for any tuple of ample divisors $(H_1, \ldots, H_{n-2})$.

Recent works of Campana and P{\u a}un (\cite{CP13}, \cite{CP16}) have generalized some parts of Miyaoka's results, 
showing in particular that if $X$ is a smooth projective variety with $K_X$ pseudoeffective,  then $\Omega^1_X$ is semipositive 
with respect to any \emph{movable} class $\alpha \in \overline\Mov_1(X)$ (see Definition~\ref{def:MOVE}).

Our first result is a natural generalization of the inequality (\ref{eq:MiyChern}) to the setting of pairs with movable log-canonical divisors.
\begin{theorem}
\label{thm:BabyPositiveC2}
Let $X$ be a normal projective threefold that is smooth in codimension two and $D$ a
reduced effective divisor such that $(X,D)$ 
has only isolated lc singularities. If $(K_X+D)\in \overline\Mov^1(X)$, then for any ample divisor $A$, the 
inequality $$c_2\bigl((\Omega^1_X\log(D))^{**}\bigr) \cdot A  \geq 0$$ holds.

\end{theorem}

The second result is another generalization of an inequality established by Miyaoka \cite{Miyaoka87}, 
which is sometimes referred to as the Miyaoka--Yau inequality. 

\begin{theorem}\label{thm:MYIntro}
Let $X$ be a normal projective threefold that is smooth in codimension two 
and $D$ a reduced effective divisor such that $(X,D)$ has only isolated lc singularities. 
If $(K_X+D)\in \overline\Mov^1(X)$, then 

$$c_1^2\bigl((\Omega^1_X\log(D))^{**}\bigr) \cdot A \leq 3c_2\bigl((\Omega^1_X\log(D))^{**}\bigr) \cdot A,$$
for any ample divisor $A$.
\end{theorem}

Theorem~\ref{thm:MYIntro} will also be established for pairs $(X,D)$ of dimension three 
with isolated singularities (see Theorem~\ref{HD}). 

There are two main ingredients in the proof of the above inequalities. The first one is a restriction result 
for semistable sheaves 
with respect to certain movable curves. This is described in Section \ref{sect:section3_restriction}. The second 
component involves the semipositivity of the \emph{orbifold cotangent sheaves} 
for certain mildly singular pairs and is treated in Section \ref{sect:section4_semipositivity}.

The rest of the paper is devoted to two applications of Theorems~\ref{thm:BabyPositiveC2} 
and~\ref{thm:MYIntro}. 
The first is concerned with the so-called effective non-vanishing conjecture.

\begin{conjecture}[Effective non-vanishing conjecture of Kawamata]
\label{conj:Effective}
Let $Y$ be a normal projective variety and $D_Y$ an effective $\bR$-divisor such that 
$(Y,D_Y)$ is klt. Let $H$ be an ample, or more generally big and nef, divisor such that 
$(K_Y+D_Y+H)$ is Cartier and nef. Then $H^0(X, K_Y+D_Y+H)\neq 0$.
\end{conjecture}

Using Theorem~\ref{thm:BabyPositiveC2}, in Section~\ref{sect:section6_Nonvanishing}, 
we obtain a simple proof of the following weak
version of Conjecture~\ref{conj:Effective} in dimension three.

\begin{theorem}[Non-vanishing for canonical threefolds]
\label{thm:Main}
Let $Y$ be a normal projective threefold with only canonical singularities. 
Let $H$ be a very ample divisor. If $(K_Y+H)$ is a nef and Cartier divisor and not numerically trivial,
then $H^0(Y, K_Y+H)\neq 0$.
\end{theorem}

We note that Theorem~\ref{thm:Main} is stated in \cite{HOR12} under the weaker assumption that $H$ is a 
nef and big Cartier divisor. The proof relies on an inequality similar to that of Theorem \ref{thm:BabyPositiveC2} but 
under the weaker assumption that the 
first Chern class is \emph{nef in codimension one}. It seems that there is a gap in the proof of that inequality, 
but the author kindly informs us that one can get rid of this assumption and use only the classical result of Miyaoka,
where $c_1$ is assumed to be nef (see the inequality (\ref{eq:MiyChern})).

\

A second application is given in section \ref{sect:LV} vis-\`a-vis Lang--Vojta's conjectures on subvarieties of
varieties of general type:

\vspace{5mm}

\begin{changemargin}{0.5cm}{0.5cm}
\emph{{\bf Geometric Lang--Vojta conjecture:} In a projective variety of general type $X$, subvarieties that are not 
of general type are contained in a proper algebraic subvariety of $X$.} 
\end{changemargin}

\vspace{5mm}

In particular, a variety of general type should have only finitely many codimension one subvarieties 
that are not of general type. We partially establish this conjecture in the setting of the following theorem.

\begin{theorem}\label{thm:LV1}
Let $X$ be a normal projective $\bQ$-factorial threefold such that $K_X\in \Mov^1(X)$. If $X$ is of general type then $X$ has only a finite number of movable codimension one, normal subvarieties $D$ 
verifying the following conditions. 

\begin{enumerate}
\item The subvariety $D$ has only canonical singularities.  
\item The anticanonical divisor $-K_D$ is pseudoeffective.
\item The pair $(X,D)$ has only isolated lc singularities. 
\end{enumerate}
In particular, there are only finitely many such Fano, Abelian and Calabi-Yau subvarieties.
\end{theorem}

Here, by a variety of general type, we mean a normal variety whose resolution has a big canonical bundle.

We remark that---in the smooth setting---a stronger version of Theorems~\ref{thm:LV1}  
and~\ref{thm:BabyPositiveC2} has been
claimed in~\cite{LuMi}, where the authors establish these results under the weaker assumption 
that $(K_X+D)$ is pseudoeffective. Unfortunately the arguments in~\cite{LuMi}
are not complete. We refer to Remark~\ref{rem:LMProblems} for a detailed discussion of these
problems.

\subsection{Acknowledgements}
The authors would like to thank S\'ebastien Boucksom, Junyan Cao, Paolo Cascini, Andreas H\"oring, Steven Lu, Mihai P\u{a}un for fruitful discussions, and the anonymous referee for  all the remarks that greatly improved the reading of this article.

\section{Basic definitions and background}

\subsection{Movable cone}
We introduce the \emph{movable} cone of divisors; one of the important cones of divisors that 
is ubiquitous in birational geometry.

Let $X$ be a normal projective variety and $D$ a $\bQ$-divisor on $X$. The stable base locus of $D$ is defined by
$$\bB(D):=\bigcap_m \Bs(|mD|).$$
The restricted base locus is given by
$$\bB_{-}(D)=\bigcup_{A \text{ ample}} \bB(D+A).$$

\begin{definition}[Movable cone of divisors]
\label{def:Move}
Let $\N^1(X)_{\bQ}$ be the space of numerical classes of divisors over $\bQ$; the Neron--Severi space.
The movable cone $\overline \Mov^1(X) \subset \N^1(X)_{\bQ}$ is the closure of the convex cone
$\Mov^1(X)$ generated by the classes of all effective divisors $D$ such that $\bB_{-}(D)$ has no divisorial components.
\end{definition}

The following inclusions now follow from the definitions. 

$$
\underbrace{\Amp(X)}_{\text{ample cone}} \;\;\;  \subset \;\;\;  \underbrace{\mathrm{Nef}(X)}_{\text{nef cone}} \;\;\; 
\subseteq \;\;\ \overline \Mov^1(X) \;\;\;  \subseteq \underbrace{\overline \Eff (X) }_{\text{pseudoeffective cone}} 
\subset \N^1(X)_{\bQ}.
$$

The following proposition gives a more geometric picture of Definition~\ref{def:Move}.

\begin{proposition}[\cite{Boucksom04}, Proposition 2.3]\label{prop:Fujita}
Given any $\alpha$ in the interior of $\Mov^1(X)$, there is a birational map $\phi: Y \to X$ 
and an ample divisor $A$ on $Y$ such that $[\phi_*A]=\alpha.$
\end{proposition}

\subsection{Stability with respect to movable $1$-cycles}
Now we introduce the notion of movable curves with respect to which a slope stability theory for sheaves
can be formulated. 
\begin{definition}\label{def:MOVE}
A class $\gamma \in \N_1(X)$ is movable if there is a projective birational morphism 
$\pi: \wtilde X \to X$ and a set of ample classes $H_1, \ldots, H_{n-1}$ in $\N^1(\wtilde X)_{\bQ}$
such that $\gamma$ is equal to the class of  $\pi_* ( H_1 \cdot \ldots \cdot H_{n-1})$.
We define $\Mov_1(X)$ to be 
the convex cone generated by such $1$-cycles and denote its closure in $\N_1(X)_{\bQ}$ by $\overline\Mov_1(X)$.
\end{definition}

Movable classes form a natural setting for the notion of stability of 
coherent sheaves (see \cite{CP11} and \cite{GKP15}). We shall now recall the basic definitions and 
properties.

\begin{notation}[determinant sheaves]
Throughout this paper by $\det(\sE)$ we mean the reflexive hull of
the determinant sheaf of $\sE$.
\end{notation}

\begin{notation}\label{not:weilclass} Let $X$ be a normal projective variety and $\sF$ a coherent sheaf 
on $X$ of rank $r$. Let $D$ be a Weil divisor in $X$ such that $\det(\sF)\cong \sO_X(D)$. 
When $D$ is $\bQ$-Cartier, we set $[\sF]$ to denote the numerical class $[D]\in \N^1(X)_{\bQ}$ of $D$.
\end{notation}

\begin{definition}
Assume that $X$ is normal, $\bQ$-factorial and projective, let $\gamma \in \Mov_1(X)$. The slope of a coherent sheaf $\sE$ of rank $r$
with respect to $\gamma$ is defined by
$$\mu_\gamma(\sE):=\frac{1}{r} \cdot  [\sE] \cdot \gamma \;\; \in \bQ.$$

\end{definition}

\begin{definition}\label{def:SSS}
We say that a torsion free sheaf $\sE$ is semistable with respect to $\gamma$, if $\mu_\gamma(\mathcal{F}) \leq \mu_\gamma(\sE)$ 
for any coherent subsheaf $0 \subsetneq \mathcal{F} \subset \sE$.
\end{definition}

\begin{proposition}[\cite{GKP15}, Corollary 2.27]
\label{GKPprop}
Let $X$ be a normal, $\bQ$-factorial, projective variety, $\gamma \in \Mov_1(X)$ and $\sE$ a torsion free sheaf. 
There exists a unique Harder--Narasimhan (or HN, for short) filtration $(\sE, F^{\HN})$, i.e. a 
filtration $0=\sE_0 \subsetneq \sE_1 \subsetneq \dots \subsetneq \sE_r=\sE$, where 
each quotient $\mathcal{Q}_i:=\sE_i/\sE_{i-1}$ is 
torsion-free, $\gamma$-semistable, and where the sequence of slopes $\mu_\gamma(\mathcal{Q}_i)$ is strictly decreasing.
\end{proposition}

\begin{remark}
We note that, by definition, the intersection of $\gamma\in \Mov_1(X)$ with any effective divisor 
is strictly positive. 
Therefore, to have a reasonable notion of stability, one works with elements of $\Mov_1(X)$ instead of those of 
its closure. 
\end{remark}

\begin{notation}[Polarization]
Given a projective normal variety $X$, let $D_1, \ldots, D_k \in \N^1(X)_{\bQ}$.
By $(D_1, \ldots, D_k)$ we denote the element of $\rm H^{2k}(X)$ defined 
by $D_1\cdot \ldots \cdot D_k$. 
\end{notation}

\subsection{Chern classes for singular spaces}
For any $i\in \bN$, let $X$ be a quasi-projective variety 
that is smooth in codimension $i$. 
For every coherent sheaf $\sF$ on $X$, by using a finite projective resolution
of $\sF|_{X_{\reg}}$, we can define the i-th Chern class $c_i(\sF|_{X_{\reg}})$
as an element of the Chow ring $\rm{A}^i(X_{\reg})$,
cf.~\cite{Ful98}. On the other hand, with $Z:= X\backslash X_{\reg}$, 
there is a natural exact sequence of 
Abelian groups 

$$
0 \longrightarrow  \rm A^i(Z)  \longrightarrow \rm A^i(X) \longrightarrow \rm A^i(X_{\reg}) \longrightarrow 0.
$$
Therefore, if $\codim_X(Z)>i$, then we have $\rm A^i(X)\cong \rm A^i(X_{\reg})$.
In particular, for normal varieties, $c_1(\sF)$ can be defined as an element of  $\rm A^1(X)$ 
as the class of the (unique) extension of $c_1(\sF|_{\reg})$. 
Similarly, if $X$ is smooth in codimension two, we can define $c_2(\sF) \in \rm A^2(X)$.
Consequently, assuming that $X$ is projective, $c_2(\sF)$ induces 
a multilinear form on 
$$\underbrace{\N^1(X)_{\bQ}\times \ldots \times \N^1(X)_{\bQ}}_{\text{$(n-2)$-times}},$$ 
where $n=\dim X$.

\begin{remark}[Non-$\bQ$-factorial case]
\label{QQ}
If for every torsion free subsheaf $\sF\subseteq \sE$, the class of $\gamma\in \Mov^1(X)$
has a representative by a smooth curve $C\subset X_{\reg}$
such that $\sF|_C$ is locally free, then the $\bQ$-factoriality assumption 
in Definition~\ref{def:SSS} is redundant. In this case we can define $\gamma$-slope of 
$\sF$ by 

$$
\frac{1}{r} c_1 (\sF)\cdot [C] = \frac{1}{r} \cdot \deg (\sF|_C) ,
$$ 
where $[C]\in A^{n-1}(X)$.
One can relax the above---rather stringent---assumptions on 
$C$ but that would be unnecessary for our purposes in the current article.
We note that the $\bQ$-factoriality assumption in Proposition~\ref{GKPprop} is redundant 
if $\gamma$ is a $1$-cycle class of this form.
\end{remark}

\medskip

\subsection{$\bQ$-twisted sheaves}

It will be quite useful in the sequel to work in the more general setting of $\bQ$-twisted sheaves as introduced in
\cite{Miyaoka87}.

\begin{definition}[$\bQ$-twisted sheaves]\label{def:Qtwists}
A $\bQ$-twisted sheaf is a pair $\sE \langle B\rangle$, where $\sE$ is a coherent sheaf 
and $B$ is a $\bQ$-Cartier divisor. 
\end{definition}

We now recall the usual formulas for Chern classes of $\bQ$-twisted locally free sheaves.

\begin{definition}
 For a $\bQ$-twisted locally-free sheaf $\sE \langle B\rangle$ of rank $r$ on a normal 
 quasi-projective variety we have
$$c_1(\sE \langle B\rangle):=c_1(\sE)+rc_1(B),$$
$$c_2(\sE \langle B\rangle):=c_2(\sE)+(r-1)c_1(\sE)\cdot c_1(B)+\frac{r(r-1)}{2}c_1(B)^2.$$
\end{definition}

\begin{notation}
In the setting of Notation~\ref{not:weilclass},
for any $\bQ$-Cartier divisor $A$, we set
$[\sF\langle A \rangle]= [\sF]+r\cdot [A]$.
\end{notation}

For $\bQ$-factorial normal projective varieties we can define a notion of slope
stability for $\bQ$-twisted sheaves with respect to $\gamma\in \Mov^1(X)$ in the natural way.
Moreover, from 
the definition it follows that 
$\sE$ is $\gamma$-semistable (or stable) if and only if $\sE\langle B \rangle$ is 
$\gamma$-semistable (resp. stable) as $\bQ$-twisted sheaf (see also~\cite[Rem.~6.4.8]{Laz04-II}). 
In particular
the following inequality follows from the well-known Bogomolov--Gieseker inequality
for smooth projective surfaces~\cite{Bog79}.

\begin{proposition}[Bogomolov--Gieseker inequality for semistable $\bQ$-twisted 
sheaves]\label{prop:QBog} 
Take $S$ to 
be a smooth projective surface. Let $\sE\langle B \rangle$ be 
a $\bQ$-twisted locally-free sheaf on $S$ of rank $r$ and $A\in \Amp(X)_{\bQ}$. If $\sE\langle B\rangle$ 
is semistable with respect to $A$, then $\sE\langle B\rangle$ verifies Bogomolov--Gieseker
inequality
\begin{equation}\label{eq:BogIneqSurface}
2r\cdot c_2(\sE\langle B \rangle) - (r-1) \cdot c_1^2(\sE\langle B \rangle)\geq 0.
\end{equation}

\end{proposition}

\begin{definition}[Semipositive sheaves]
\label{def:SP}
Let $X$ be a normal, $\bQ$-factorial, projective variety and $\gamma \in \Mov_1(X)$.
A torsion-free sheaf $\sE$ is said to be semipositive with respect 
to $\gamma$, if for every torsion-free, quotient sheaf $\sF$ of
$\sE$, we have 
$[\sF] \cdot \gamma \geq 0$.

\end{definition}

\begin{remark}
Similar to the case of stability, 
if for every torsion free quotient $\sF$ of $\sE$, the class of
$\gamma$ has a representative by a smooth projective curve $C$
as in Remark~\ref{QQ}, then
the $\bQ$-factoriality assumption in Definition~\ref{def:SP} is not necessary. In
this case the semipositivity assumption is given by

$$ 
c_1(\sF) \cdot [C] = \deg(\sF|_C) \geq 0.
$$
For a $\bQ$-factorial projective variety $X$ the two definitions coincide. 
\end{remark}

The semipositivity property for sheaves also naturally extends to the setting of $\bQ$-twisted
sheaves. We say $\sE\langle B \rangle$ is semipositive with respect to $\gamma$,
if $[\sF\langle B \rangle] \cdot \gamma \geq 0$, for all torsion free quotient sheaves $\sF$.

\medskip

\subsection{Orbifold basics}
Following the terminology of Campana \cite{Ca04}, an \emph{orbifold} is simply a pair $(X,D)$, consisting of a normal 
quasi-projective variety and a boundary divisor 
$D=\sum d_i\cdot D_i$, where $d_i= (1-b_i/a_i)\in [0,1]\cap \bQ$.
We follow the usual convention that when $d_i=1$ we have ``$a_i= \infty$".
Throughout this article all pairs $(X,D)$ will be of this form.
As such, we frequently refer to them simply as \emph{pairs}. 
When $X$ is projective, we refer to $(X,D)$ as above as a \emph{projective pair}.
We say $(X,D)$ is log-smooth, if $X$ is smooth and $D$ has simple normal 
crossing support. 

Our main aim is now to define a notion of cotangent sheaf, adapted to a pair. 
To this end, and since we will not be exclusively working with smooth varieties, 
we will need a notion of pull-back for Weil divisors (that are not necessarily 
$\bQ$-Cartier). We denote the group of Weil divisors by $\WDiv(X)$ and set
$\WDiv(X)_{\bQ} := \WDiv(X)\otimes \bQ$ to denote the group of $\bQ$-Weil divisors.

\begin{definition}[Pull-back of Weil divisors]\label{def:Wpull-back}
 Let $f:Y\to X$ be a finite morphism between quasi-projective 
normal varieties $X$ and $Y$. We define the pull-back $f^*(D)$ of a $\bQ$-Weil divisor $D\subset X$ by the 
Zariski closure of $f^*(D|_{X_{\reg}})$.

\end{definition}

\begin{notation}
Given a pair $(X,D)$, with $D=\sum d_i \cdot D_i$, we use the following 
notations:

$$
\lfloor D \rfloor : = \sum \lfloor d_i \rfloor \cdot D_i,
$$

$$
\lceil D \rceil : = \sum \lceil d_i \rceil \cdot D_i,
$$
where $\lfloor d_i \rfloor$ denotes the round-down and $\lceil d_i \rceil$ the round-up.
\end{notation}

\begin{definition}[Adapted and strongly adapted morphisms]
\label{def:adapted}
Let $(X,D)$ be an orbifold.
A finite, surjective morphism $f: Y\to X$ is called \emph{adapted} (to $D$) if, 
$f^*D$ is an integral Weil divisor and $f$ is unramified at the generic point of $\lfloor D \rfloor$. 
We say that a given adapted morphism $f: Y \to X$ is \emph{strictly adapted}, if 
we have $f^*D_i = a_i \cdot D_i'$, for some Weil divisor $D_i'\subset Y$.
Furthermore, we call a strictly adapted morphism $f$, \emph{strongly adapted}, if
the branch locus of $f$ only consists of $\supp \bigl(  D-\lfloor D \rfloor  + A  \bigr)$,
where $A$ is a general member of a basepoint free linear system 
on $X$. 

\end{definition}

\begin{remark}\label{rem:KawConstructions}
For a log-smooth pair $(X,D)$, 
the existence of a strongly adapted morphism $f: Y \to X$ 
was established by Kawamata, cf.~\cite[Prop.~4.1]{Laz04-II}.
A similar strategy can be applied to construct strongly adapted 
morphisms $f: Y \to X$ when all the irreducible components of $D$ are 
$\bQ$-Cartier; in particular when $X$ is assumed to be $\bQ$-factorial. 
Alternatively, one can use the following more general statement, which 
follows from Kawamata's original result.
\end{remark}

\begin{proposition}
\label{prop:KAW}
Let $D\subset X$ be a prime divisor on a normal quasi-projective variety. 
For every $m\in \bN$ there is a normal variety $Y$, a finite morphism $f:Y\to X$ and a Weil divisor 
$D_Y\subset Y$ such that 

\begin{enumerate}
\item $f^*D = m\cdot D_Y$, and that
\item \label{prop:GG} the branched locus of $f$ consists of $D$ and a general member of a basepoint free linear system.
\end{enumerate}

\end{proposition}

\begin{proof}
Let $\pi: (\wtilde X, \wtilde D) \to (X,D)$ be a log-resolution. 
By \cite[Prop.~4.1]{Laz04-II} we know that there is a morphism
$\wtilde f: \wtilde Y\to \wtilde X$ of smooth quasi-projective varieties
such that $\wtilde f^* \wtilde D = m D_{\wtilde Y}$, for some Weil divisor 
$D_{\wtilde Y}\subset \wtilde Y$. Define $g:= \pi \circ  \wtilde f$.
Now, let $\mu:\wtilde Y\to Y$ be the birational morphism and $f:Y\to X$
the finite map arising from Stein factorization of $g$.
After normalization, if necessary, the morphism $f:Y\to X$ 
satisfies the required properties\footnote{For an effective, $\pi$-exceptional divisor $E$ and ample divisor $A\subset X$, 
the ample divisor in Kawamata's
construction should be taken to be of the form $(\pi^*A - a\cdot E)$, with $a\in \bQ^+$ sufficiently 
small, so that $(\pi^*A - a\cdot E)$ is ample. This guarantees that Property~\ref{prop:GG} is satisfied.}. 
\end{proof}

The following lemma is useful in the course of the arguments in 
Section~\ref{sect:section4_semipositivity}.
Its proof follows directly from the construction of adapted morphisms. Nevertheless,
for the reader's convenience, we include a brief argument. 

\begin{lemma}
\label{lem:Adapt}
Let $(X,D)$ be a pair with $\lfloor D \rfloor  =0$, and $g: Y\to X$ any finite 
morphism of normal varieties. 
There is an adapted morphism 
$h: Z\to  (X,D)$ factoring through $g$ and a morphism $r:Z\to Y$.

\end{lemma}

\begin{proof}
Setting $D= \sum(1- b_i/a_i) \cdot D_i$, for every $i$, let $n_j\in \bN$ be the integer for which we have 

$$
g^* D_i =  \sum^{k_i}_{j(i)=1} n_j \cdot D_{ij} ,
$$
for some $k_i \in \bN$. Define $m_i:= \rm{lcm} (n_1, \ldots, n_{k_i})$.
By Proposition~\ref{prop:KAW}, for each $D_{ij}\subset Y$ we can 
construct a morphism $r_{ij}: Z_{ij}\to Y$ such that 

$$
r_{ij}^* (D_{ij}) = (\frac{m_i}{n_j} \cdot a_i ) \cdot B_{ij},
$$
for some $B_{ij}\subset Z_{ij}$.
Let $r: Z\to Y$ denote the composition of all such maps. Then, by construction we have 

\begin{align*}
r^*  (g^* D_i)      &  =  \sum_{j=1}^{k_i}  n_j \cdot r^*(D_{ij})   \\
         &  = (a_i \cdot m_i) \sum_{j=1}^{k_i}  B_{ij}  ,
\end{align*}
as required. 
\end{proof}

\begin{notation}\label{not:adapt} Let $f: Y\to X$ be a morphism adapted to $D$, 
where $D=\sum d_i\cdot D_i$, $d_i=1-\frac{b_i}{a_i}\in (0,1]\cap \bQ$.
For every irreducible component $D_i$ of $(D-\lfloor D\rfloor)$, let $\{D_{ij}\}_{j(i)}$ be the 
collection of prime divisors that appear in $f^{*}(D_i)$. We define new divisors in $Y$ by 
\begin{flalign}\label{not}
&D_Y^{ij}:=b_i\cdot D_{ij} \\
&D_{f}:=f^*(\lfloor D\rfloor).
\end{flalign}
\end{notation}

\medskip

Now, let us explain how to define the cotangent sheaf of an orbifold (or a pair).

\begin{definition}[Orbifold cotangent sheaf]\label{def:adaptedsheaf} In the situation of 
Notation~\ref{not:adapt}, denote 
$Y^\circ$ to be the log-smooth locus of the pair $(Y,\sum D_{ij}+D_{f})$ and define 
${D_Y^{ij}}^\circ:=D_Y^{ij}|_{Y^\circ}$. 
Set $\Omega^1_{(Y^\circ,f,D)}$ to be the kernel of the sheaf morphism 

$$
(f|_{Y^\circ})^*\bigl(\Omega_X\log (\lceil D \rceil)\bigr) \longrightarrow
\bigoplus \limits_{i,j(i)}\sO_{{D_Y^{ij}}^\circ}$$
induced by the natural residue map. We define the orbifold cotangent sheaf 
$\Omega_{(Y,f,D)}^{[1]}$ 
by the coherent extension 
$(i_{Y^\circ})_*(\Omega_{(Y^\circ,f,D)}^1)$, where $i_{Y^\circ}$ is the natural inclusion. 
We define the \emph{orbifold tangent sheaf} $\sT_{(Y,f,D)}$ by $(\Omega_{(Y, f ,D)}^{[1]})^*$.

\end{definition}

\section{Restriction results for semistable sheaves}
\label{sect:section3_restriction}

Let $h=(H_1,\dots, H_{n-1})$ be a tuple of ample divisors on a normal
projective variety $X$ of dimension $n$ and $\sE$ a torsion free sheaf.
A theorem of Mehta--Ramanathan \cite{MR82} states that if $m$ is large enough
and $Y \in |mH_{n-1}|$ is a generic hypersurface, then the maximal destabilizing
subsheaf of $\sE|_Y$ is the restriction of the maximal destabilizing
subsheaf of $\sE$.

It is natural to try to extend this restriction theorem to movable polarization.
Unfortunately, in general, such results are not valid for movable curve. 
For example, when $X$ is a projective $K3$ surface then its cotangent bundle
$\Omega^1_X$ is not pseudoeffective, which gives rise to the existence of movable curves
for which the restriction theorem does not hold (cf.~\cite[Sect.~7]{BDPP}).

 In this section, we will prove a restriction theorem for some movable curves
 (see Proposition~\ref{prop:restrict} below). 
 The following lemmas will serve as key technical ingredients in the proof of this 
 result.  
 
 \begin{set-up}\label{SetUpLemma}
 Let $\pi: \wtilde S\to S$ be a birational morphism of smooth projective surfaces 
$\wtilde S$ and $S$.
Let $\wtilde A_{\wtilde S}\subset \wtilde S$ be an ample divisor 
and define $P_S:=[ \pi_*(\wtilde A_{\wtilde S})] \in \N^1(S)_{\bQ}$.
  \end{set-up}
 
 \begin{lemma}[Induced destabilizing subsheaves of small rank on higher birational models. I]\label{lem:rank2}
 In the setting of Set-up~\ref{SetUpLemma},
 let $\sG_S$ be a locally free sheaf on $S$ of rank two.
Assume that $\sF_S \subset \sG_S$ is a saturated and properly destabilizing subsheaf of $\sG_S$.
If $\wtilde \sB \subset \pi^*\sG_S$ is the maximal destabilizing subsheaf of $\pi^*\sG_S$, 
then $\big(\pi_*(\wtilde \sB)\big)^{**} \cong \sF_S$.

 \end{lemma}

\begin{proof}
The proof is a direct consequence of the assumptions made 
 on the slopes of $\wtilde \sB$ and $\sF_S$. More precisely, if we consider the exact sequence 
 $$
 0  \longrightarrow \overline{\pi^*\sF_S} \longrightarrow \pi^*\sG_S \longrightarrow \wtilde \sQ \longrightarrow 0,
 $$
 where $\overline{\pi^*\sF_S}$ is the saturation of $\pi^*\sF_S$ in $\pi^*\sG_S$,
 we have the slope inequality $\mu_{\wtilde A_{\wtilde S}}(\wtilde \sQ)  <  \mu_{\wtilde A_{\wtilde S}} (\pi^*\sG_S)$. 
 This implies that the 
 induced map from $\wtilde \sB$ to $\wtilde \sQ$ is zero, otherwise we 
 get $\mu_{\wtilde A_{\wtilde S}}(\wtilde \sB) < \mu_{\wtilde A_{\wtilde S}}(\wtilde \sQ)$,
 which is absurd.
 It thus follows that
 there is an injection 
 $$
 \wtilde \sB  \hooklongrightarrow \overline{\pi^*\sF_S}
 $$
 But the slope of $\wtilde \sB$ is maximal. Therefore 
 $\wtilde \sB \cong \overline{\pi^*\sF_S}$ and this proves the claim. 
 \end{proof}

 \medskip

 \begin{lemma}\label{lem:injection}
 In the setting of Set-up~\ref{SetUpLemma} let $\sE_S$ be a locally free sheaf of rank three on $S$.
 Define $\wtilde \sG_{\wtilde S} \subset \pi^*\sE_S$ to be the maximal destabilizing subsheaf with respect to $\wtilde A_{\wtilde S}$.
 If $\rank(\wtilde \sG_{\wtilde S}) =2$, then
 for every saturated, properly destabilizing subsheaf $\sF_S \subset \sE_S$  of rank one with respect to $P_S$,
 there is an injective morphism 
 $$
 \sF_S \hooklongrightarrow  \sG_S, 
 $$
 where $\sG_S : = \big( \pi_*\wtilde \sG_{\wtilde S}  \big)^{**}$.
 \end{lemma}
 
\begin{proof}
Let $\overline{\pi^*\sF_S}$ denote the saturation of $\pi^*\sF_S$ in $\pi^*\sE_S$
and set $\wtilde \sQ$ to be the torsion free quotient $\pi^*\sE_S/ \overline{\pi^*\sF_S}$,
whose slopes satisfies the inequality 

 \begin{equation}\label{eq:now}
 \mu(\wtilde \sQ) <  \mu(\pi^*\sE_S).
 \end{equation}
As $\rank(\sF_{S}) =1$, there is a nontrivial morphism $\sigma: \wtilde \sG_{\wtilde S} \to \wtilde \sQ$.
 
Now, if $\sigma$ is injective, then $\mu(\wtilde \sG_{\wtilde S}) \leq \mu(\wtilde \sQ)$.
 It then follows from the inequality (\ref{eq:now}) that 
 $$
 \mu(\wtilde \sG_{\wtilde S})  < \mu(\pi^*\sE_S),
 $$
 a contradiction. 
 Therefore $\wtilde \sK: = \mathrm{Im}(\sigma) \subset \wtilde \sQ$ is a rank 
 one subsheaf, giving rise to the commutative diagram of exact sequences:

 $$
 \xymatrix{
 0  \ar[rr] &&  \overline{\pi^*\sF_S}  \ar[rr] && \pi^*\sE_S  \ar[rr] &&   \wtilde \sQ \ar[rr] && 0 \\
 0 \ar[rr] &&  \wtilde \sM  \ar[u]   \ar[rr] &&  \wtilde \sG_{\wtilde S}  \ar[rr]^{\sigma} \ar[u]  &&  \wtilde \sK  \ar[rr]  \ar[u] && 0.
 }
 $$
 
 \begin{claim}\label{claim:destable}
 $\mu_{\wtilde A_{\wtilde S}}(\wtilde \sK)  > \mu_{\wtilde A_{\wtilde S}}(\wtilde \sQ)$.
 \end{claim}
 
 \noindent
 \emph{Proof of Claim~\ref{claim:destable}.}
 Aiming for a contradiction, assume that $\mu (\wtilde \sK) \leq \mu (\wtilde \sQ)$.
 Now, as $ \wtilde \sG_{\wtilde S}$ is $\wtilde A_{\wtilde S}$-semistable, the inequality 
 
 \begin{equation}\label{eq:obvious}
  \mu ( \wtilde \sG_{\wtilde S})  \leq   \mu (\wtilde \sK)
 \end{equation}
 holds. On the other hand, we have 
 
 \begin{equation}\label{eq:3rd}
 \mu(\wtilde \sK)  \leq \mu (\wtilde \sQ)  <  \mu(\overline{\pi^*\sF_S}),
 \end{equation}
 where the last inequality follows from (\ref{eq:now}); that is, the fact that $\overline{\pi^*\sF_S}\subset \pi^*\sE_S$
 is properly destabilizing. But (\ref{eq:obvious}) and (\ref{eq:3rd}) lead to the inequality 
 
 $$
 \mu(\wtilde \sG_{\wtilde S})  <  \mu( \overline{\pi^*\sF_S}),
 $$
 which contradicts the assumption on $\wtilde \sG_{\wtilde S}$ having the maximal slope.
 This finishes the proof of the claim.

 We now consider the saturation of $\wtilde \sK$, which we denote by $\wtilde \sK_1$, 
 as a properly destabilizing subsheaf of $\wtilde \sQ$ with the resulting exact 
 sequence of sheaves
 
 $$
\xymatrix{
0 \ar[r] & \ar[r]  \wtilde \sK_1 \ar[r]^{\tau}  &  \wtilde \sQ \ar[r]  &  \wtilde \sA \ar[r]  & 0
}
 $$
 that are locally free in codimension one.
 Let $\wtilde \sB$ be the kernel of the induced surjection $\gamma: \pi^*\sE_S \longrightarrow \wtilde \sA$
 and 
 
 $$
 \xymatrix{
 0 \ar[r]  &  \wtilde \sB  \ar[r]  &  \pi^*\sE_S  \ar[r]^{\gamma}  &   \wtilde \sA  \ar[r] &  0,
 }
 $$
 the corresponding exact sequence.
 
 As $\wtilde \sK$ generically coincides with $\ker(\tau) = \wtilde \sK_1$, the induced  
 map $\wtilde \sG_{\wtilde S}  \longrightarrow \wtilde \sA$ is zero. Therefore there
 is an injection 
 
 $$
 \wtilde \sG_{\wtilde S} \hooklongrightarrow \wtilde \sB.
 $$
Now, since $\wtilde \sG_{\wtilde S}$ is the maximal destabilizing subsheaf 
(with respect to the ample divisor $\wtilde A_{\wtilde S}$), it follows that 
$\wtilde \sG_{\wtilde S}$ and $\wtilde \sB$ are isomorphic in codimension one.

On the other hand, from the sequence 

$$
\xymatrix{
0 \ar[r] &  \overline{\pi^*\sF_S}  \ar[r]  & \pi^*\sE_S  \ar[r]^g \ar@/^6mm/[rr]^{\gamma}  &  \wtilde \sQ \ar[r]  &  \wtilde \sA 
},
$$
we can see that there is an injection $\overline{\pi^*\sF_S} = \ker(g) \hooklongrightarrow \ker(\gamma) = \wtilde \sB$.
Noting that $\sF_S \subset \sE_S$ is saturated and thus reflexive, this implies that 
$\sF_S$ injects into  $(\pi_* \wtilde \sG_{\wtilde S})^{**}$.
\end{proof}

 \begin{lemma}[Induced destabilizing subsheaves of small rank on higher birational models. II]
\label{lem:FujitaSubsheaf}
In the situation of Set-up~\ref{SetUpLemma},
let $\sE_S$ be a $P_S$-unstable locally free sheaf of rank three on $S$. 
Assume that $\sE_S$ contains a saturated properly destabilizing  
subsheaf $\sF_S$ of rank one. 
Let $\wtilde \sG_{\wtilde S}  \subset \pi^*\sE_S$ be the maximal $\wtilde A_{\wtilde S}$-destabilizing 
subsheaf of $\pi^*\sE_S$ and define $\sG_S: = (\pi_* \wtilde \sG_{\wtilde S})^{**}$. Then,

\begin{enumerate}
\item \label{item:ref} either $\sG_S$ is a properly destabilizing subsheaf of $\sE_S$,
 
\item or \label{item:filter1} there is a nontrivial morphism 

$$
(*)  \; \; \; \; \;   \pi^* \sF_S  \to  \wtilde \sG_{\wtilde S}
$$
inducing an injection $\sF_S \to \sG_S$, whose image is properly destabilizing.

\item \label{item:filter2} Or the Harder--Narasimhan filtration of $\pi^*\sE_S$ has two steps.
With $\wtilde \sD_{\wtilde S}:=  F^{\HN}_2(\pi^*\sE_S)$ and $\sD_S:= (\pi_* \wtilde \sD_S)^{**}$,
 there is an injection $\sF_S \to \sD_S \subset \sE_S$. Moreover, we either have 

$$
\mu_{P_S}(\sD_S)  >  \mu_{P_S}(\sE_S)
$$ 
or $\sD_S$ is not semistable, with the image of $\sF_S$ in $\sD_S$ being 
a properly destabilizing subsheaf.

\end{enumerate}

\end{lemma}

\medskip

\begin{proof}
We exclude Item~\ref{item:ref} by making the assumption that 

\begin{equation}\label{eq:Assump}
\mu_{P_S} (\sG_S) \leq  \mu_{P_S} (\sE_S) .
\end{equation}
Assume further that $\rank (\sG_S)=2$. Then, according to Lemma~\ref{lem:injection}
 we have an injection 
 \begin{equation}\label{eq:FinalInject}
   \sF_S \hooklongrightarrow \sG_S. 
   \end{equation}
Using (\ref{eq:Assump}) we can see that the image of $\sF_S$ under the map (\ref{eq:FinalInject})
  properly destabilizes $\sG_S$.

Now, if $\rank (\sG_S) =1$, consider the exact sequence 
 
 $$
\xymatrix{
0 \ar[r] &  \wtilde \sG_{\wtilde S} \ar[r] & \pi^*\sE_S \ar[r]   \ar[r]^{\nu} & \wtilde \sC \ar[r]  & 0.  
}
$$
Using (\ref{eq:Assump}) again, we can see that there is an injection 
$\pi^* \sF_S \hooklongrightarrow \wtilde \sC$.
Then, the inequalities 

\begin{align*}
 \mu(\pi^*\sF_S)  & >  \mu (\pi^*\sE_S), \; \; \; \text{since $\sF_S \subset \sE_S$ is destabilizing}\\
                           & >  \mu(\wtilde \sC), \; \; \; \text{as $\wtilde \sG_{\wtilde S} \subset \pi^*\sE_S$ is destabilizing}
\end{align*}
imply that the image of $\pi^*\sF_S$ under $\nu$ destabilizes $\wtilde \sC$.
Therefore, there is a second step $\wtilde \sD_{\wtilde S}$ in the HN-filtration of $\pi^*\sE_S$.

\begin{claim}\label{claim:2nd}
There is an injection $\pi^*\sF_S  \hooklongrightarrow  \wtilde \sD_{\wtilde S}$.
\end{claim}

Assuming Claim~\ref{claim:2nd} for the moment, we proceed to finish 
the proof of the lemma. We note that once the injection in Claim~\ref{claim:2nd}
 exists, then we have $\sF_S \hooklongrightarrow \sD_S$.
 Now, either $\mu(\sD_S)> \mu(\sE_S)$ or 

 $$\mu(\sD_S)\leq \mu(\sE_S).$$
 If the latter inequality holds, then (the image of) $\sF_S$ properly 
destabilizes $\sD_S$ and this finishes the proof of the lemma.

\medskip

It now remains to establish Claim~\ref{claim:2nd}.

\noindent
\emph{Proof of Claim~\ref{claim:2nd}.} 
The first observation is that $\wtilde \sD_{\wtilde S} \subset \pi^*\sE_S$ is destabilizing.
To see this, we consider the two exact sequences 

$$
\xymatrix{
0 \ar[r] &  \wtilde \sG_{\wtilde S}  \ar[r]  &  \wtilde \sD_{\wtilde S} \ar[r] &  \sQ' \ar[r]    &  0, 
}
$$

$$
\xymatrix{
0 \ar[r]  &  \wtilde \sD_{\wtilde S}  \ar[r]  &  \pi^* \sE_S  \ar[r]^j  &  \wtilde \sA  \ar[r] &  0,
}
$$
with the two sheaves $\sQ'$ and $\wtilde \sA$ being the successive quotients of the HN-filtration. 
By the definition of HN-filtration, we know that 

\begin{equation}\label{1ST}
\mu(\sQ')  >  \mu(\wtilde \sA).
 \end{equation}
On the other hand, as $\wtilde \sG_{\wtilde S}$ is the maximal destabilizing 
subsheaf, from the first sequence we have 

\begin{equation}\label{2ND}
\mu (\sQ') <  \mu(\wtilde  \sD_{\wtilde S})  <  \mu (\wtilde \sG_{\wtilde S}).
 \end{equation}
Combining (\ref{1ST}) and (\ref{2ND}) we have 

$$
\mu(\wtilde \sA)  <  \mu(\wtilde \sD_{\wtilde S}).
$$
From the second sequence it now follows that 

$$
\mu(\pi^*\sE_S) <  \mu(\wtilde \sD_{\wtilde S}), 
$$
i.e. $\wtilde \sD_{\wtilde S}$ destabilizes $\pi^*\sE_S$.
Consequently $\pi^*\sF_S$ projects to zero via the morphism $j$
so that there is a map $\pi^*\sF_S \to  \wtilde \sD_{\wtilde S}$,
as required. 
\end{proof}

 \medskip

 \begin{lemma}[Induced destabilizing subsheaves of small rank on higher birational models. III]
 \label{lem:last}
 In the setting of Lemma~\ref{lem:FujitaSubsheaf}, let $\sE_S$ be a $P_S$-unstable locally free
 sheaf of rank $3$ on $S$. Assume that $\sE_S$ contains a saturated proper destabilizing 
 subsheaf $\sF_S$ of rank $2$. Let $\wtilde \sN_{\wtilde S} \subset \pi^* \bigwedge^2 \sE_S$ be the maximal destabilizing 
 subsheaf of $\pi^* \bigwedge^2 \sE_S$ and $\sN_S: = (\pi_* \wtilde \sN_{\wtilde S})^{**}$. Then,

  \begin{enumerate}
 \item either the subsheaf $\sN_S = \bigl( ( \pi_*\wtilde \sN_{\wtilde S})^{**}  \bigr)$ destabilizes $\bigwedge^2 \sE_S$, 
 \item or we have an injection $\bigwedge^2 \sF_S \hooklongrightarrow \sN_S$ with the image of $\bigwedge^2\sF_S$ 
 being properly destabilizing, 
 \item or the second step $\wtilde \sA_{\wtilde S}$ of the HN-fitlration of $\bigwedge^2 \pi^*\sE_S$ 
 descends to a destabilizing subsheaf $\sA_S$ of $\bigwedge^2 \sE_S$, 
 \item or $\bigwedge^2 \sF_S$ maps into 
 $\sA_S$ inducing a properly destabilizing subsheaf.
 \end{enumerate}

 \end{lemma}

 \begin{proof}
  By using the fact that 
 $$
  \mu(\bigwedge^2 \sE_S) = 2\cdot \mu (\sE_S)   \;\;\;\;\;\;\;  \text{and} \;\;\;\;\;\;\;  \mu(\bigwedge^2 \sF_S) = 2\cdot \mu(\sF_S),
 $$
 we can see that $\bigwedge^2 \sF_S \subset \bigwedge^2 \sE_S$ is a 
 properly destabilizing, saturated subsheaf. Now, as $\rank(\bigwedge^2 \sF_S) =1$ and $\rank(\bigwedge^2 \sE_S)=3$, 
 Lemma~\ref{lem:FujitaSubsheaf} applies and settles the proof. 
  \end{proof}

 \medskip
 
 \begin{remark} The main objective of the above lemmas is to show that 
once we have an unstable bundle $\sE_S$ of rank at most $3$ on $S$,  
there is a term of a HN filtration on $\wtilde S$ that ``descends" to 
a destabilizing subsheaf of $\sE$ (or $\bigwedge^2 \sE$). Now, either the sheaf on $\wtilde S$ is 
the first or second step of the HN-filtration of $\pi^*\sE$ 
(or $\wtilde \sN_{\wtilde S}\subset \bigwedge^2\pi^*\sE$), 
that is 
\begin{equation}\label{eq:explain}
   \mu_{P_S} (\sG_S) > \mu_{P_S}(\sE)       \;\;\;\;\;\;\;\;\;   \text{(resp. $\mu_{P_S}(\sN_S) > \mu_{P_S}(\bigwedge^2 \sE_S)$)},
\end{equation}
or it is the maximal destabilizing subsheaf of one of the two steps of the HN filtration for $\pi^*\sE_S$. 
To be more precise, for example when rank of a properly detabilizing subsheaf $\sF$ of $\sE_S$ is one, 
if the inequality (\ref{eq:explain}) does not hold, then the image of
$\sF_S$ in $\sG_S$ is destabilizing (see~\ref{item:filter1}). Therefore, according to Lemma~\ref{lem:rank2}, the maximal 
destabilizing subsheaf $\wtilde \sB \subset \pi^*{\sG_S}$ descends to the sheaf $(\pi_*\wtilde \sB)^{**}$ on $S$ that is
isomorphic to the saturation of (the image of)
$\sF_S$ in $\sG_S$. In particular we have 

$$
\mu_{P_S} \big( ( \pi_*\wtilde \sB )^{**} \big)  > \mu_{P_S}(\sE_S).
$$ 
One can argue similarly for the case of Item~\ref{item:filter2} or when $\rank(\sF_S)=2$.
Uniqueness of such sheaves 
on $\wtilde S$ will play a crucial role in the proof of Proposition~\ref{prop:restrict}.

\end{remark}

 \

The next proposition is the main result in this section, proving a restriction 
result for semistable sheaves with respect to a particular set of movable classes. 
As we will see later in Section~\ref{sect:section5_PseudoEffectivity}, 
these classes naturally arise in the 
context of positivity problems for second Chern classes.

\begin{proposition}[A restriction result for movable classes]
\label{prop:restrict}
Let $X$ be a normal projective threefold that is smooth in codimension two.
Let $P\in \overline \Mov^1(X)$ and $H_1, H_2\in \Amp(X)_{\bQ}$.
Let $\sE$ be a torsion free sheaf on $X$ of rank $3$. There exists a positive 
integer $M_1$ such that for all sufficiently divisible integers $m_1\geq M_1$, 
there is a Zariski open subset $V_{m_1}\subset |m_1\cdot H_1|$ 
for which the following properties holds.
 
\begin{enumerate}
\item \label{item:prop1} Every member $S\in V_{m_1}$ is smooth, irreducible and 
is contained in $X_{\reg}$.

\item \label{item:prop2} The restriction $\sE|_S$ is torsion free.

\item \label{item:prop3} The divisor $P|_S$ is nef.

\item \label{item:prop4} For every such $S$, there exists $M_2\in \bN^+$ such that 
every sufficiently divisible integer $m_2\geq M_2$ gives rise to a Zariski open subset 
$V_{m_2}\subset |m_2 \cdot (P+H_2)|_S |$, where every $\gamma \in V_{m_2}$ is 
a smooth, irreducible curve in $S$ such that $\sE|_{\gamma}$ is locally free and verifies the following property:

\noindent
(*) The formation of the $\HN$-filtration of $\sE$ with respect to $(H_1, P+H_2)$ 
commutes with restriction to $\gamma$, i.e. $\HN_{\bullet}(\sE)|_{\gamma} =  \HN_{\bullet}(\sE|_{\gamma})$.

\end{enumerate}

\end{proposition}

\begin{proof} Let $\pi:\wtilde X\to X$ be the birational morphism and $\wtilde X$ the smooth 
projective variety with ample divisor $\wtilde A\subset \wtilde X$ 
associated to 
the Fujita approximation of the big movable divisor $P+H_2$ in the interior of $\Mov^1(X)$, 
i.e. 
$$
\pi_*\wtilde A = [(P+H_2)]
$$
(cf.~Proposition~\ref{prop:Fujita},~\cite[Sect.~11.4]{Laz04-II} and~\cite[Chapt.~III]{Nak04}). 

Now, let $N_1\in \bN^+$ be a sufficiently large and divisible integer such that for every $n_1\geq N_1$, 
there are open subsets $U_{n_1}\subset |n_1\cdot \pi^*H_1|$ and $\wtilde U_{n_1}\subset |n_1\cdot \wtilde A|$,
where for every subscheme $\wtilde S:= D_{n_1}$ and $\wtilde \gamma:=\wtilde D_{n_1}\cap D_{n_1}$, with 
$D_{n_1}\in U_{n_1}$ and $\wtilde D_{n_1}\in \wtilde U_{n_1}$, we have: 

\begin{enumerate}
\item \label{item:1n1} Both $\wtilde S$ and $\wtilde \gamma$ are smooth and irreducible. 
\item \label{item:2n1} The restrictions $(\pi^{[*]}\sE)|_{\wtilde S}$ and $\big( \bigwedge^{[2]} \pi^{[*]}\sE \big)|_{\wtilde S}$ are locally free.
\item \label{item:3n1} The HN-filtration of $\pi^{[*]}\sE$ with respect to $(\pi^*H_1, \wtilde A)$ 
verifies: $\HN_{\bullet}\bigl((\pi^{[*]}\sE\bigr)|_{\wtilde S})=\HN_{\bullet}(\pi^{[*]}\sE)|_{\wtilde S}$.
In addition, the same property holds for $\bigwedge^{[2]} \pi^{[*]} \sE$.
\end{enumerate}

\noindent
The positive integer $N_1$ exists, thanks to Bertini theorem and Langer's restriction 
theorem for stable sheaves, cf.~\cite{Langer04a}. 

\vspace{4 mm}

\noindent \emph{Step.~1. (Reflexivity assumption).}
By the Bertini theorem and~\cite[Thm.~12.2.1]{EGA4-2}, and as $P\in \overline \Mov^1(X)$, 
there exists a positive integer $N_2$ such that 
for every sufficiently divisible $n_2\geq N_2$ there exists a Zariski open 
subset $V_{n_2}\subset |n_2\cdot H_1|$, where every $S\in V_{n_2}$
satisfies the three Properties~\ref{item:prop1},~\ref{item:prop2} and~\ref{item:prop3}
\footnote{Here we are using the fact, which is a consequence of Fujita's approximation for movable divisors, 
that any codimension one movable class 
is \emph{nef in codimension one}, that is its restriction to a sufficiently general surface is nef.}. 
We can also ensure that 
every $S\in V_{n_2}$ is transversal to the exceptional centre of $\pi$.  
Furthermore, as $P|_S$ is nef, we can find $N_3\in \bN^+$ such that for each 
sufficiently divisible $n_3\geq N_3$, the general member of $\gamma\in |n_3\cdot (P+H_2)|_S|$
is smooth and is contained in an open subset of $X$ over which the HN-filtration of 
$\sE$ (with respect to $(H_1, P+H_2)$) is a filtration of $\sE$ by locally-free sheaves.
Therefore, to prove that Property (*) is verified by $\gamma$, we 
may assume, without loss of generality, that $\sE$ is reflexive
and therefore for a suitable choice of $S$, its restriction $\sE|_S$ is locally free.

\vspace{3 mm}

\noindent \emph{Step.~2. (Construction of $\gamma$).} 
Let $M_2\geq N_3$ be a sufficiently large and divisible integer such that 
for every $m_2\geq M_2$ there exists a Zariski open subset $V_{m_2}\subset |m_2(P+H_2)|_S |$, 
where every curve $\gamma \in V_{m_2}$ is smooth and is contained in the 
complement of the exceptional center of $\pi$. 
Furthermore, 
$\sE|_{\gamma}$ is locally free, 
and if $\sE|_{\gamma}$ is not 
semistable, then $\sE_S:=\sE|_S$ is not semistable with respect to $(P+H_2)|_S$
and more generally we have $\HN_{\bullet}(\sE_S)|_{\gamma}=\HN_{\bullet}(\sE|_{\gamma})$. 
The existence of such $M_2$ is guaranteed by Bertini theorem and Mehta--Ramanathan's restriction theorem, 
cf.~\cite{MR82}. 

Summarizing these geometric constructions, by choosing sufficiently large 
$n_i$ and $m_2$ and by shrinking $V_{n_i}$ and $V_{m_2}$ if necessary, 
we have $\gamma\subset S$ and $\wtilde \gamma \subset \wtilde S$, with surjective morphisms 

$$
\pi|_{\wtilde S}: \wtilde S\to S  \;\;\;\;  \text{and} \;\;\;  \pi|_{\wtilde \gamma}: \wtilde \gamma\to \gamma,
$$
and satisfying Properties~\ref{item:1n1}, \ref{item:2n1}, \ref{item:3n1},
and those in the setting of the proposition but excluding (*).

Now, to prove the proposition, it suffices to show that if a reflexive sheaf $\sE$ is semistable 
with respect to $(H_1, P+H_2)$, then so is $\sE|_{\gamma}$. So let us now 
assume that $\sE$ is indeed semistable. The next steps are devoted to proving 
that $\sE|_{\gamma}$ is also semistable.

Aiming for a contradiction, assume that $\sE|_{\gamma}$ is not semistable. 
It follows that $(\pi^{[*]}\sE)|_{\wtilde \gamma} =  \pi^* (\sE|_{\gamma})$ is not semistable. 
By the construction of $\gamma$, this also implies that $\sE_S$ is unstable.
Therefore, $\pi^*\sE_S$ is unstable with respect to $\wtilde \gamma$ (which is numerically proportional to 
$\wtilde A|_{\wtilde S}$).
Moreover, thanks to \cite{Langer04a},
unstability of $(\pi^{[*]}\sE|_{\wtilde \gamma})$ implies that $\pi^{[*]}\sE$ is unstable (with respect to $(\pi^*H_1, \wtilde A)$).

\begin{claim}\label{claim:NotSS}
$\bigwedge^{[2]} \pi^{[*]}\sE$ is not semistable with respect to $(\pi^*H_1, \wtilde A)$.
\end{claim}

\noindent
\emph{Proof of Claim~\ref{claim:NotSS}.}
This follows directly from rank considerations. Suppose $\sM \subset \pi^{[*]}\sE$ is a saturated 
destabilizing subsheaf. 

If $\rank(\sM)=2$, then as $\mu(\bigwedge^2 \sM) = 2\cdot \mu(\sM)$
and $\mu(\bigwedge^{[2]} \pi^{[*]}\sE)  = 2\cdot \mu(\pi^{[*]}\sE)$, the 
subsheaf $\bigwedge^2 \sM \subset \bigwedge^{[2]} \pi^{[*]}\sE$ is destabilizing. 

Now, if $\rank(\sM)=1$, then $\mu(\sQ)< \mu(\pi^{[*]}\sE)$, where $\sQ$ is the torsion free quotient 
$\pi^{[*]}\sE/ \sM$. Again by using the fact that $\mu(\bigwedge^2 \sQ) =2\cdot \mu(\sQ)$, 
we find that the inequality 

$$
\mu(\bigwedge^{[2]}\sQ) <  \mu(\bigwedge^{[2]} \pi^{[*]}\sE)
$$
holds, implying that $\bigwedge^{[2]} \pi^{[*]}\sE$ is not semistable. This finishes the proof of Claim~\ref{claim:NotSS}.\qed

\medskip

Now, let $\wtilde \sG$ and $\wtilde \sN$ be the first step 
of the HN-filtration of $\pi^{[*]}\sE$ and $\bigwedge^{[2]} \pi^{[*]}\sE$, respectively
and define 

$$
\sG: = \big( \pi_*(\wtilde \sG) \big)^{**}  \;\;\;\;\;\;  \text{,} \;\;\;\;\;\;  \sN : = \big( \pi_*(\wtilde \sN)  \big)^{**}.
$$
Assuming that they exist, let $\wtilde \sD$ and $\wtilde \sA$ be the second step of the HN-filtration of
$\pi^{[*]}\sE$ and $\bigwedge^{[2]} \pi^{[*]}\sE$ and set

$$
\sD: = \big( \pi_*(\wtilde \sD) \big)^{**}  \;\;\;\;\;\;  \text{,} \;\;\;\;\;\;  \sA : = \big( \pi_*(\wtilde \sA)  \big)^{**}.
$$

Let $m_1\in \bN$ be a sufficiently divisible integer, verifying the inequality 
$m_1 \geq M_1:=\max\{ N_1, N_2 \}$, and such that there
is an open subset $V_{m_1}\subseteq | m_1 \cdot H_1 |$ for which we have the following property. 
After shrinking $V_{m_1}$, if necessary, for every $S\in V_{m_1}$ (defined in Steps.~1 and 2), we have

\begin{enumerate}
\item \label{p1} $\wtilde S:= \pi^*(S) \in U_{m_2}$,

\item $\wtilde \sG|_{\wtilde S}$, $\wtilde \sD|_{\wtilde S}$, $\wtilde \sN|_{\wtilde S}$ and $\wtilde \sA|_{\wtilde S}$ are locally 
free, 
\item \label{miss} $S$ does not intersect the singular loci of $\sE$, $\sG$, $\sD$, $\sN$ and $\sA$, and

\item \label{item:ISOM} we have $( \pi_* (\wtilde \sG|_{\wtilde S}) )^{**}\cong  
(\pi_* \wtilde \sG)^{**}|_S$ and the same holds for $\wtilde \sD$, $\wtilde \sN$ and $\wtilde \sA$.
\end{enumerate}

\noindent 
\emph{Step.~3. (Extension of maximal destabilizing subsheaves).} 
We are now in the setting where we can apply Lemmas~\ref{lem:rank2}, \ref{lem:FujitaSubsheaf} and~\ref{lem:last}.
Let $\wtilde \sG_{\wtilde S}$ and $\wtilde \sD_{\wtilde S}$ be the first and second steps of the HN-filtration of 
$\pi^*{\sE_S}$, assuming that the latter exists. By construction, using Property~\ref{item:3n1} together with
Properties~\ref{p1} and~\ref{miss}, there are isomorphism 

$$
\wtilde \sG_{\wtilde S}\cong \wtilde \sG|_{\wtilde S} \; \; , \; \;  \wtilde \sD_{\wtilde S} \cong \wtilde \sD|_{\wtilde S}.
$$

Let us first assume that $\sE_S$ contains a saturated destabilizing subsheaf
$\sF_S$ of rank one. According to Lemma~\ref{lem:FujitaSubsheaf}, one of 
the locally free sheaves $\sG_S :=  \pi_*  (\wtilde \sG_{\wtilde S})^{**}$ 
or $\sD_S: = \pi_*(\wtilde \sD_{\wtilde S})^{**}$

\smallskip

(*) either destabilizes $\sE_S$, 

(**) or it is not semistable and admits an injection from $\sF_S$ with a properly destabilizing image.

\smallskip

We identify $\sF_S$ with 
its image under $\sF_S \hooklongrightarrow \sG_S$ (respectively, $\sF_S$ with its image under 
$\sF_S \hooklongrightarrow \sD_S$).

Now, if (*) holds, then we have our desired contradiction since by \ref{item:ISOM}
the subsheaf $(\pi_*\wtilde \sG)^{**}\subset \sE$ or $(\pi_*\wtilde \sD)^{**}\subset \sE$ 
is properly destabilizing.

So assume that (**)  is true. We observe that by our choice of $S$ (Property \ref{miss}) we have 
$(\pi|_S)^* \sG_S =  \big( \pi^{[*]}  \sG \big)|_{\wtilde S}$ (and $(\pi|_S)^* \sD_S = (\pi^{[*]}\sD)|_{\wtilde S}$).
We can now apply Lemma~\ref{lem:rank2}. More precisely, 
if $\sF_S \subset \sG_S$ (or $\sF_S \subset \sD_S$) is destabilizing $\sG_S$ (respectively, $\sD_S$), 
then, according to Lemma~\ref{lem:rank2},
the maximal destabilizing 
subsheaf $\wtilde \sL_{\wtilde S}$ of $\pi^*(\sG_S)$ verifies the isomorphism 

\begin{equation}\label{eq:global}
\big((\pi|_{\wtilde S})_* (\wtilde {\sL}_{\wtilde S}))^{**} \cong  \overline{\sF_S},
\end{equation}
where $\overline \sF_S$ is the saturation of $\sF_S$ in $\sG_S$, and 
similarly when $\sF_S \subset \sD_S$ is destabilizing.

On the other hand, again by the restriction result \cite{Langer04a}, we have 

$$
\wtilde \sL|_{\wtilde S} \cong  \wtilde \sL_{\wtilde S},
$$
where $\wtilde \sL$ is the maximal destabilizing subsheaf of $\pi^{[*]}\sG$
(after adjusting the choice of $S$ and $\wtilde S$ if necessary).
Therefore, by (\ref{eq:global}) $\wtilde \sL$ descends to a destabilizing subsheaf 
of $\sE$, i.e. $(\pi_*\sL)^{**} \subset \sE$ is destabilizing, a contradiction. 
Similarly we can argue that the maximal destabilizing subsheaf of 
$\wtilde \sD$ descends to a   destabilizing subsheaf of $\sE$.

Next, we assume that $\rank(\sF_S)=2$. In this case Lemma~\ref{lem:last}
applies. The same arguments as above (this time for $\sN$, $\sA$ instead of $\sG$ and $\sD$) then shows
that $\bigwedge^2\sE$ is not semistable.
On the other hand, thanks to \cite[Thm.~4.2]{GKP15}, we know that semistable sheaves 
with respect to movable classes over normal varieties form a tensor category\footnote{This result
has an additional assumption; $\bQ$-factoriality of $X$. As we pointed out in Remark~\ref{QQ},
this condition is unnecessary in the context of this proposition.}.
As a result we again get a contradiction to the semistability 
assumption on $\sE$. 
\end{proof}

\begin{remark}[Restriction of $\HN$-filtration for $\bQ$-twisted sheaves]
\label{rem:Qrestrict} 
We note that the consequences of Proposition~\ref{prop:restrict} 
are still valid for $\bQ$-twisted torsion-free sheaves. More precisely, given a 
$\bQ$-twisted, torsion-free sheaf $\sE\langle B\rangle$ and 
$H_1, H_2 \in \Amp(X)_{\bQ}$, $P\in \overline\Mov^1(X)$, there is a complete intersection surface $S$ and 
$\gamma\subset S$, as in Proposition~\ref{prop:restrict}, such that 
$\HN_{\bullet}(\sE\langle B \rangle)|_{\gamma}=
      \HN_{\bullet}(\sE\langle B\rangle|_{\gamma})$.
To see this we can use the fact that, for every torsion free sheaf $\sF$ and Weil $\bQ$-divisor $B$, we have

$$
\HN_{\bullet} \big(  \sE \langle B \rangle \big)  =   \big( \HN_{\bullet}(\sE) \big)  \langle B \rangle ,  
$$
which follows directly from the definitions. The rest now follows from Proposition~\ref{prop:restrict}.

\end{remark}

\begin{remark}[Restriction result in higher dimensions]
Following the same arguments 
 as those of the proof of Proposition~\ref{prop:restrict}, we can remove the 
restriction on the dimension, that is the consequences of Proposition~\ref{prop:restrict}
are still valid, if $X$ is of dimension $n> 3$ and the polarization is 
$(H_1, H_2, \ldots, (P+H_{n-1}))$, for any $H_1,\ldots, H_{n-1}\in \Amp(X)_{\bQ}$,
as long as $\rank (\sE) =3$.
\end{remark}

As an immediate consequence we establish a Bogomolov--Gieseker inequality
for ($\bQ$-twisted) sheaves of small rank that are semistable with respect to movable classes 
of the form that appear in Proposition~\ref{prop:restrict}. Although we do not 
use this inequality in the rest of the paper, we find it to be of independent 
interest.  

\begin{proposition}[Bogomolov--Gieseker inequality in higher dimensions]\label{prop:highBog-ineq}
Let $X$ be an $n$-dimensional, normal projective variety that is smooth in codimension two 
and $\sE\langle B\rangle$ a $\bQ$-twisted, 
reflexive sheaf of rank at most equal to $3$ on 
$X$. If $\sE\langle B\rangle$ is semistable with respect to 
$(H_1, H_2, \ldots, (P+H_{n-1}))$,
where $H_i\in \Amp(X)_{\bQ}$ and $P\in \overline\Mov^1(X)$, then 
$$\bigl(2r \cdot c_2(\sE\langle B\rangle) - 
      (r-1)\cdot  c_1^2(\sE\langle B\rangle)\bigr)\cdot H_1 \ldots \cdot H_{n-2}\geq 0.$$

\end{proposition}

\begin{proof} This is an immediate consequence of the restriction result in Proposition~\ref{prop:restrict}
and Remark~\ref{rem:Qrestrict}
together with Proposition~\ref{prop:QBog}.
\end{proof}

\section{Semipostivity of adapted sheaf of forms}
\label{sect:section4_semipositivity}

In~\cite{CP16} Campana and P\u{a}un remarkably prove that the orbifold 
cotangent sheaf of a log-smooth pair $(X,D)$ is semipositive with respect 
to movable curve classes on $X$ (see Theorem~\ref{thm:CP16} below).
See Definition~\ref{def:SP} for the definition of this notion of semipositivity.
Currently it is not clear if this result can be easily extended to the case 
of singular pairs. In the present section we show that, for a special subset 
of movable classes, the generalization to singular pairs
can be achieved by essentially reducing to the smooth case.

\begin{theorem}[\protect{Orbifold semipositivity with respect to movable classes, cf.~\cite[Thm.~1.2]{CP16}}]
\label{thm:CP16} 
Given a log-smooth pair $(X,D)$, if $(K_X+D)$ is pseudoeffective, then for any movable 
class $\gamma\in \Mov_1(X)$ and 
any adapted 
morphism $f:Y\to X$, where $Y$ is smooth, the orbifold cotangent sheaf $\Omega_{(Y,f,D)}^1$ is 
semipositive with respect to 
$f^*(\gamma)$\footnote{Here we are following the notation of \cite{CP16} for ``pullback" of movable $1$-cycles. 
Since in the current paper
we are only concerned with those cycles that are defined by divisors, we have forgone the exact 
definition of this notion.}.
\end{theorem}

In the next proposition we slightly refine Theorem~\ref{thm:CP16} for a class of movable $1$-cycles that 
we call \emph{complete intersection $1$-cycles}. 
As we will see later in 
Section~\ref{sect:section5_PseudoEffectivity}, such classes appear naturally in our
treatment of the pseudoeffectivity of $c_2$.

\begin{definition}[Complete intersection movable classes]
We say that $\gamma\in \Mov_1(X)$ 
is a complete intersection movable 
$1$-cycle, if there are classes $B_1,\ldots, B_{n-1}\in \N^1(X)_{\bQ}$ such that $\gamma$ 
is numerically equivalent to the cycle defined 
by $(B_1\cdot \ldots \cdot B_{n-1}) \in \N_1(X)_{\bQ}$. 

\end{definition}

\begin{proposition}[A refinement of the orbifold semipositivity result]\label{prop:CP16refined}
Let $(X,D)$ be a log-smooth pair and $\gamma \in \Mov_1(X)$ a complete intersection movable 
cycle. If $(K_X+D)$ is pseudoeffective, then for any strictly adapted morphism $g:Z\to X$ (see Definition~\ref{def:adapted}),
 $\Omega_{(Z,g,D)}^{[1]}$ is semipositive with respect to 
$g^*\gamma$.

\end{proposition}

\begin{proof}
Assume that $Z$ is not smooth, otherwise the claim follows from the arguments of Campana
and P\u{a}un, cf.~\cite{CP16}. Let $D=\sum d_i \cdot D_i$, where $D_i$ are prime divisors and 
$d_i= 1 - (b_i/a_i) \in [0,1]\cap \bQ$. 
By assumption, for every $D_i$, we have $g^*(D_i)=a_i\cdot D_{Z,i}$, for some 
$D_{Z,i}\in \WDiv(Z)$. 

Now, set $f:Y \to X$ to be a strongly adapted morphism (Definition ~\ref{def:adapted}), where,  
thanks to Kawamata's construction, cf.~\cite[Prop.~4.1.12]{Laz04-II}, 
the variety $Y$ is smooth. 
Let $W$ be the normalization of fibre product 
$Y\times_X Z$ with the resulting commutative diagram

$$
  \xymatrix{
    W \ar[rr]^{v}  \ar[d]_{u}  \ar[drr]^{h}  &&  Z  \ar[d]^{g} \\
       Y \ar[rr]^{f}       && X.
  }
  $$
 Aiming for a contradiction, assume that $\Omega_{(Z,g,D)}^{[1]}$ is not semipositive 
 with respect to $g^*\gamma$, that is there exists a reflexive subsheaf 
 $\sG_Z \subset \Omega_{(Z,g,D)}^{[1]}$ such that 
 \begin{equation}\label{eq:negative1}
 \bigl(  \gamma^*(K_X+D) - [\sG_Z]  \bigr) \cdot g^*\gamma <0.
 \end{equation}
 We consider $v^{[*]}(\sG_Z) \subset \Omega_{(W,h,D)}^{[1]}$.
 As $\gamma$ is, numerically, a complete intersection cycle, we can use the 
 projection formula to conclude that 
 \begin{equation}\label{eq:negative2}
 \bigl( h^*(K_X+D) - [v^{[*]} \sG_Z] \bigr) \cdot h^*\gamma <0,
 \end{equation}
 which implies that $\Omega_{(W,h,D)}^{[1]}$ is not semipostive with respect to 
 $h^*\gamma$. Now, let $\Omega_{(W,h,D)}^{[1]} \onto \sF_W$ be the torsion 
 free quotient with the minimum slope
 with the kernel $\sG_W$:
 
 \begin{equation}\label{eq:minimal1}
 0 \to \sG_W \to \Omega_{(W,h,D)}^{[1]}  \to  \sF_W  \to 0.
 \end{equation} 
 Assuming that $u:W \to Y$ is Galois, let $G:=\Gal(W/Y)$. Notice that by the construction of $f$, we have 
 $\Omega_{(W,h,D)}^{[1]}=u^*(\Omega_{(Y,f,d)}^1)$.
 Now, as the inclusion $\sG_W \subset \Omega_{(W,h,D)}^1$ 
 is saturated, and since $\sG_W$ is a $G$-subsheaf (thanks to its
 uniqueness), according to~\cite[Thm.~4.2.15]{MR2665168}
 or~\cite[Prop.~2.16]{GKPT15}, 
 there exists a reflexive subsheaf $\sG_Y \subset \Omega^1_{(Y,f,D)}$ 
 such that $u^{[*]}(\sG_Y) \cong \sG_W$.
 
 Now, by applying the $G$-invariant section functor $u_*(\cdot )^G$ to the exact sequence (\ref{eq:minimal1})
 we find that
 \begin{equation}
 0 \to \sG_Y  \to  \Omega_{(Y,f,D)}^1 \to \bigl( u_*(\sF_W)  \bigr)^G \to 0.
 \end{equation}
 But, by the projection formula, it follows that
 
 $$
 \big(   f^*(K_X+D) - [\sG_Y]  \big) \cdot f^*\gamma \leq 0,
 $$
 i.e. $\Omega_{(Y,f,D)}^1$ is not semipositive with respect to $f^*\gamma$, which contradicts 
 Theorem~\ref{thm:CP16}.
 
 For the case where $u$ is not Galois, we can consider the Galois closure 
 $u' :  W' \xrightarrow{\sigma} W \xrightarrow{u} Y$ and repeat the above 
 argument for $\sigma^{[*]}(\Omega^{[1]}_{(W, h, D)})$ instead of $\Omega^{[1]}_{(W, h, D)}$.
 \end{proof}

\vspace{2 mm}

The next proposition is the extension of Theorem~\ref{thm:CP16} to a special class of complete 
intersection, movable $1$-cycles on a mildly singular $X$.

\begin{proposition}[Semipositivity for mildly singular pairs]
\label{prop:SingCP}
Given a projective pair $(X,D)$, assume that
$(K_X+D)$ is pseudoeffective. 
Let $H_1 \ldots , H_{n-1} \in \Amp(X)_{\bQ}$ 
and $P\in \overline\Mov^1(X)$.
Then, 
for any strictly adapted morphism $f:Y\to X$, the orbifold cotangent sheaf 
$\Omega_{(Y,f,D)}^{[1]}$ is semipositive with respect to 
$f^*(H_1, \ldots, H_{n-2}, P+H_{n-1})$, if $(X,D)$ verifies one of the following assumptions. 

\begin{enumerate}
\item \label{item:sing1}$(X,D)$ has only klt singularities.
\item \label{item:sing2} $D$ is reduced (i.e. $\lfloor D \rfloor =0$) and $(X,D)$ has only lc singularities. 
\end{enumerate}

\end{proposition}

\begin{proof}
Assume that the assumption~\ref{item:sing1} holds. 
Let $\pi: (\wtilde X, \wtilde D)\to (X,D)$ be a log-resolution and $\wtilde Y$
the normalization of the fibre product $Y\times_X \wtilde Y$
with the commutative diagram
$$
  \xymatrix{
    \wtilde Y \ar[rr]^{\wtilde f}  \ar[d]_{\wtilde \pi}  && \wtilde X  \ar[d]^{\pi} \\
       Y \ar[rr]^{f}       && X, 
  }
  $$
where $\wtilde \pi:\wtilde Y \to Y$ and $\wtilde f: \wtilde Y \to Y$ are the 
naturally induced projections.

For simplicity, and as the arguments are identical in higher dimensions, 
we only deal with the case when $\dim X=3$.
Denote $H_{Y,i}=f^*(H_i)$, for $i\in \{ 1, 2\}$ and $P_Y=f^*(P)$.

Now, aiming for a contradiction, assume that $\Omega_{(Y,f,D)}^{[1]}$ is not semipositive 
with respect to $(H_{Y,1}, P_{Y}+H_{Y,2})$.
This implies that there exists a saturated subsheaf $\sG \subset \sT_{(Y, f, D)}$
such that $[\sG]\cdot (H_{Y,1}, P_{Y}+H_{Y,2}) > 0$.
Define $\wtilde \sH :=  (\wtilde \pi^{[*]} \sH) \cap \sT_{(\wtilde Y, \wtilde f, D)}$. 
Let $m$ be a sufficiently large positive integer such that the $1$-cycle 
$\gamma\in \Mov_1(Y)$, that is 
numerically equivalent to the cycle defined by $m^2(H_{Y,1}, P_{Y}+H_{Y,2})$,
is away from the exceptional centre of $\wtilde \pi$. 
Existence of such $\gamma$ in particular guarantees that 

\begin{equation*}\label{eq:NegativeQuot}
 [\wtilde \sH] \cdot \wtilde \pi^* (H_{Y,1}, P_{Y}+H_{Y,2})  > 0.
\end{equation*}
In other words there exists a torsion-free quotient sheaf 

\begin{equation}\label{eq:ineq-first}
\Omega_{(\wtilde Y, \wtilde f, \wtilde D)}^{[1]} \onto \wtilde \sF 
\end{equation}
on $\wtilde Y$ such that 
$\deg(\wtilde \sF|_{\wtilde \gamma}) <0$, where $\wtilde \gamma:= \wtilde\pi^{-1}(\gamma)$.

\noindent
Now, let us consider the logarithmic ramification formula 
  \begin{equation*}
    K_{\wtilde X}+\wtilde D=\pi^*(K_X+D)+\sum a_i\cdot E_i - \sum b_i\cdot E_i', 
    \end{equation*}
where $a_i\in \bQ^+$, and, because of the assumptions on the singularities, $b_i \in (0,1) \cap \bQ$. 
Define $\wtilde G:=\sum b_i\cdot E_i'$
and let $\wtilde h : Z \to \wtilde X$ be the morphism adapted to $(\wtilde X, \wtilde D+\wtilde G)$, factoring 
through $\wtilde f: \wtilde Y \to \wtilde X$

$$
 \begin{xymatrix}{
     Z  \ar@/^5mm/[rrrr]^{\wtilde h} \ar[rr]_{r} && \wtilde Y
     \ar[rr]_{\wtilde f} && \wtilde X} ,
    \end{xymatrix} 
 $$ 
 as in Lemma~\ref{lem:Adapt}.
Set $B_Z:= \wtilde h^*( \pi^*(H_1, P + H_2))$ and $B_{\wtilde Y}:= \wtilde f^*(\pi^*(H_1, P + H_2))$.  
Now, let $\sG_{\wtilde Y}$ 
be the kernel of the sheaf morphism (\ref{eq:ineq-first}) so that 

\begin{equation}\label{eq:negativekernel}
\bigl( \wtilde f^*(K_{\wtilde X}+\wtilde D)  - [\sG_{\wtilde Y}] \bigr) \cdot B_{\wtilde Y} <0 .
\end{equation}

As $\gamma$ is away from the exceptional centre of $\wtilde \pi$ and since $\wtilde G$ is supported 
on the exceptional locus of $\pi$, we have

 \begin{align*}
 \wtilde h^*(K_{\wtilde X}+\wtilde D+ \wtilde G)\cdot B_Z
        &= \wtilde h^*(K_{\wtilde X}+\wtilde D)\cdot B_Z  \\
        &= r^*(\wtilde f^*(K_{\wtilde X}+\wtilde D))\cdot B_Z.
 \end{align*}
As a result, for the inclusion 
$r^{[*]}(\sG_{\wtilde Y})\subset \Omega_{(Z,\wtilde h, \wtilde D+\wtilde G)}^{[1]}$, we find that

 \begin{align*}
 \bigg(   \Big[   \Omega_{(Z, \wtilde h, \wtilde D+\wtilde G)}^{[1]}  \Big]  - r^{[*]}\sG_{\wtilde Y}  \bigg)
         \cdot B_Z &= \Big(  r^*\big(  \wtilde f^*(K_{\wtilde X} + \wtilde D) \big) - r^{[*]}\sG_{\wtilde Y} \Big) \cdot B_Z \nonumber \\
         &= (\deg r) \Big( \wtilde f^*(K_{\wtilde X} +\wtilde D) - [\sG_{\wtilde Y}]  \Big) \cdot B_{\wtilde Y} \\
         &< 0, \quad\quad\quad\quad\quad \text{by Inequality~\ref{eq:negativekernel},}
 \end{align*}
 contradicting Proposition~\ref{prop:CP16refined}.
 
 Finally, if the assumption~\ref{item:sing2} holds, the proof follows from simply considering the 
 ramification formula as above and using Theorem~\ref{thm:CP16}.
 \end{proof}

\section{Pseudoeffectivity of the orbifold $c_2$}
\label{sect:section5_PseudoEffectivity}

In~\cite{Miyaoka87} Miyaoka famously proved that the second Chern class $c_2$ of a generically 
semipositive sheaf with nef determinant is pseudoeffective.
Thanks to his result on the semipositivity of 
cotangent sheaves, Miyaoka then established the pseudoeffectivity of 
$c_2(X)$ for any minimal model $X$. Our aim in this section is to 
generalize this result to the case of pairs $(X,D)$ with movable 
$(K_X+D)$ (Corollary~\ref{cor:pseff-c_2}).

\begin{proposition}[Pseudoeffectivity of $c_2$ for semipositive sheaves]\label{prop:c2} 
Let $X$ be a normal projective, threefold 
with isolated singularities and $A_1\in \Amp(X)_{\bQ}$. Then, the 
inequality 
$$
c_2(\sE)\cdot  A_1 \geq 0
$$
holds for any reflexive sheaf $\sE$ of rank $r$ verifying the following properties. 
\begin{enumerate}
\item $[\sE]\in \overline\Mov^1(X)$.
\item \label{SPAssump} For any $A_2\in \Amp(X)_{\bQ}$, the sheaf $\sE$ is semipositive with respect
to $(A_1, [\sE] + A_2)$.
\end{enumerate}

\end{proposition}

\begin{proof} Let $c$ be any positive integer. Consider the $\bQ$-twisted 
reflexive sheaf $\sE\langle \frac{1}{c} \cdot H\rangle$.
For the choice of polarization $(A_1,[\sE\langle \frac{1}{c}\cdot H\rangle])$, 
the assumptions of Proposition~\ref{prop:restrict} are
satisfied, for all $c$.

Now, let $S$ be the complete intersection surface 
defined in Proposition~\ref{prop:restrict} (see also Remark~\ref{rem:Qrestrict}) 
so that, using the assumption \ref{SPAssump} with $A_2:= \frac{r}{c}H$, the restriction 
$\sE_S\langle \frac{1}{c}\cdot H_S\rangle:=(\sE\langle \frac{1}{c} \cdot H \rangle)|_S$ 
is semipositive with respect to 

$$\beta:= c_1(\sE_S\langle \frac{1}{c}\cdot H_S\rangle) = ([\sE]+\frac{r}{c} \cdot [H])|_S.$$

Following the arguments of Miyaoka, we now consider two cases based 
on the stability of $\sE_S\langle \frac{1}{c}\cdot H_S\rangle$.

First, we consider the case where $\sE_S\langle \frac{1}{c}\cdot H_S\rangle$ 
is semistable with respect to $\beta$. Here, the semipositivity  of $c_2$ follows from 
Bogomolov--Gieseker inequality for $\bQ$-twisted locally-free sheaves 
(Proposition~\ref{prop:QBog}).

Now, we assume that $\sE_S\langle \frac{1}{c} \cdot H_S\rangle$ 
is not semistable with respect to $\beta$. Let 
    \begin{equation}\label{eq:HNFilter}
       0 \neq \sE_S^1\langle \frac{1}{m}\cdot H_S\rangle \subset \ldots 
          \subset \sE_S^t\langle \frac{1}{c}\cdot H_S \rangle
          =\sE_S\langle \frac{1}{c} \cdot H_S\rangle
              \end{equation}
be the $\bQ$-twisted HN-filtration of $\sE_S\langle 
\frac{1}{c} H_S\rangle$. 
Denote the semistable, torsion-free, $\bQ$-twisted sheaves 
$$\sE_S^i\langle \frac{1}{c} \cdot H_S\rangle / \sE_S^{i-1}\langle \frac{1}{c}\cdot  H_S\rangle$$ 
of rank $r_i$ by $\sQ^i_S\langle \frac{1}{c}\cdot  H_S \rangle$ and let $\overline{\sQ}_S^i
\langle \frac{1}{c}\cdot H_S \rangle $ denote its reflexivization. As the second 
Chern character $ch_2(\cdot)$ is additive, we have 

 \begin{align}\label{eq:BogQuot}
    2\cdot c_2(\sE_S\langle \frac{1}{c} \cdot H_S\rangle) - 
    c_1^2(\sE_S\langle \frac{1}{c} \cdot H_S\rangle)
    &= \sum \bigl(  2\cdot c_2(\sQ^i_S\langle \frac{1}{c}\cdot  H_S \rangle)
    - c_1^2(\sQ^i_S\langle \frac{1}{c}\cdot  H_S \rangle)  \bigr)    \\      \nonumber  
    & \geq \sum\bigl( 2\cdot c_2(\overline{\sQ}_S^i
     \langle \frac{1}{c}\cdot H_S \rangle) -c_1^2({\sQ}_S^i
      \langle \frac{1}{c}\cdot H_S \rangle)  \bigr), \\ \nonumber
          \end{align}
where the last inequality follows from the fact that $c_2(\sQ_S^i)\geq 
c_2(\overline{\sQ}_S^i)$. Now, by applying  
the Bogomolov inequality (\ref{prop:QBog}) to each semistable, $\bQ$-twisted 
sheaf $\overline{\sQ}_S^i\langle \frac{1}{c}\cdot H_S\rangle$, we find that each term in
the right-hand side of the inequality (\ref{eq:BogQuot}) verifies the inequality 

$$
2\cdot c_2(\overline{\sQ}_S^i
     \langle \frac{1}{c}\cdot H_S \rangle) -c_1^2({\sQ}_S^i
      \langle \frac{1}{c}\cdot H_S \rangle) \geq \frac{-1}{r_i} \cdot c_1^2(\sQ_S^i \langle \frac{1}{c}\cdot H_S \rangle ).
$$
Therefore we have
  \begin{equation}\label{eq:ineqmain}
       2\cdot c_2(\sE_S\langle \frac{1}{c} \cdot H_S\rangle) - 
    c_1^2(\sE_S\langle \frac{1}{c} \cdot H_S\rangle) \geq 
      \sum \frac{-1}{r_i} \cdot c_1^2(\sQ_S^i\langle \frac{1}{c}\cdot H_S\rangle).
       \end{equation}

Next, we define the rational number $\alpha_i\in \bQ$ by the equality
   \begin{equation}\label{eq:trick}
       r_i\cdot \alpha_i  =  \frac{c_1(\sQ_S^i \langle \frac{1}{c}\cdot H_S \rangle )
           \cdot \beta }{c_1^2(\sE_S\langle \frac{1}{c}\cdot H_S\rangle)} 
           = \frac{c_1(\sQ_S^i \langle \frac{1}{c}\cdot H_S \rangle )
           \cdot \beta}{\beta^2}.
             \end{equation}
It follows that 
   \begin{equation}\label{eq:trick-2}
     \sum r_i \cdot \alpha_i =1.
     \end{equation}                          
             
Furthermore, according to the definition of $\alpha_i$, 
and by using the fact that the slopes of the quotients of the HN-filtration (\ref{eq:HNFilter}) are strictly decreasing,
we know that 
   \begin{equation}\label{eq:decrease}
      \alpha_1 > \alpha_2 > \ldots >\alpha_t \geq 0, 
         \end{equation}
where the last inequality follows from the semipositivity of $\sE_S\langle \frac{1}{c} H_S\rangle.$

Now, as $\alpha_i\geq 0$, for each $i$, the equality (\ref{eq:trick-2}) implies that $\alpha_i \leq 1$. 
On the other hand, according to the Hodge index theorem we have 
$$ -c_1^2(\sQ_S^i\langle \frac{1}{c}\cdot H_S\rangle) \geq 
  \frac{\big(c_1(\sQ_S^i \langle \frac{1}{c} \cdot H_S \rangle) \cdot \beta\big)^2 }{\beta^2} ,$$
so that 
$$ -c_1^2(\sQ_S^i \langle \frac{1}{c}\cdot H \rangle )  \geq \beta^2 (r_i\cdot \alpha_i)^2.$$
Going back to the inequality (\ref{eq:ineqmain}) we now find that
       
       \begin{align*}
    2\cdot c_2(\sE_S\langle \frac{1}{c} \cdot H_S\rangle) 
     & \geq   \beta^2(  1-  \sum r_i\cdot \alpha_i^2 )   \\        
     & \geq \beta^2(1-\alpha_1 \sum r_i\cdot \alpha_i)  && \text{by (\ref{eq:decrease})}  \\
     & = \beta^2 (1-\alpha_1) && \text{by (\ref{eq:trick-2})}  \\
     & \geq 0  && \text{as $\alpha_1 \leq 1$.}
          \end{align*}
The inequality $c_2(\sE_S)\geq 0$ now follows by taking the limit $c \to \infty$.       
\end{proof}

As an immediate consequence, we can now prove the pseudoeffectivity of 
$c_2$ for the orbifold cotangent sheaves of pairs $(X,D)$ in dimension $3$ with only mild isolated
singularities and whose $K_X+D$ is movable.

\begin{corollary}[Positivity of $c_2$ of orbifold cotangent sheaves]\label{cor:pseff-c_2}
 Let $(X,D)$ be a projective pair of dimension $3$ and with only isolated singularities. Assume that either $D$ is 
 reduced and $(X,D)$ has only lc singularities 
 or $(X,D)$ is klt.
 If $(K_X+D)\in \overline \Mov^1(X)$,  
 then for any ample divisor $A\subset X$ and strongly adapted morphism $f: Y\to X$, 
 the inequality 
$$
c_2(\Omega^{[1]}_{(Y,f, D)})\cdot f^*(A)  \geq 0
$$
 holds. 
 \end{corollary}

\begin{proof}
As $[\Omega_{(Y,f,D)}^{[1]}] = f^*(K_X+D)$, 
the corollary is a direct consequence of Proposition~\ref{prop:c2} together with 
Proposition~\ref{prop:SingCP}.
\end{proof}

We would like to point out that once we assume that $(K_X+D)$ is nef, 
then an easy adaptation of the original results of Miyaoka to the case of 
orbifold Chern classes, together with the semipositivity result of~\cite{CP14}
(see also~\cite[Thm.~5.3]{CKT16})
leads to the following theorem. 

\begin{theorem}[Positivity of orbifold $c_2$ for log-minimal models]
\label{thm:PositiveC2}
Let $(X,D)$ be a projective lc pair of dimension $n$
that is log-smooth in codimension two.
If $(K_X+D)$ is nef, then for 
any strongly adapted morphism $f:Y \to X$ , we have
$$
c_2(\Omega_{(Y,f,D)}^{[1]}) \cdot f^*(A^{n-2}) \geq 0,
$$
where $A\subset X$ is any ample divisor. 

\end{theorem}

\begin{remark}
In the above results the assumption that $(X,D)$ is log-smooth in codimension two 
can be dropped if one is willing to work woth the so-called $\mathcal Q$-Chern classes.
But in this setting $(X,D)$ would have to be klt and additional assumptions would be 
needed to guarantee that the covering $Y$ has quotient singularities in codimension two.
\end{remark}

\section{An effective non-vanishing result for threefolds}
\label{sect:section6_Nonvanishing}

The goal of this section is to prove Theorem~\ref{thm:Main}.
The main point of the strategy is to devise an effective lower bound
for $\chi(K_Y+H)$, when $Y$ is terminal (and $(Y,H)$ is lc). 
First we recall the well-known fact that Hizerbruch--Riemann--Roch Theorem
holds for locally free sheaves over projective threefolds or surfaces with only mild 
singularities 
and, for the reader's convenience, include a short proof. 

\begin{proposition}
Let $X$ be a projective variety and $L$ a Cartier divisor on $X$.

\begin{itemize}
\item \label{RR3} If $X$ is a terminal threefold, then we have

\begin{IEEEeqnarray*}{rCl}\label{eq:HRRSing}
\chi(X,  L) & = & \frac{1}{12} \cdot L \cdot (L -K_X) \cdot (2L - K_X) \\
\IEEEyesnumber
&&   + \frac{1}{12} \cdot c_2(X)\cdot L +\chi(X,\sO_X).
\end{IEEEeqnarray*}

\item \label{RRSurface} If $X$ is of dimension two and with only rational singularities, then we have 

\begin{equation}\label{RR2}
\chi (L) = \frac{1}{2}  L^2  - \frac{1}{2}  L \cdot K_X  +  \chi(X,\sO_X).
\end{equation}

\end{itemize}

\begin{proof} 
We consider the threefold case first. Let $\pi:\wtilde X\to X$ be a resolution such that $\pi^{-1}|_{X_{\reg}}$ is an 
isomorphism. Remember that, as its singularities are only terminal, $X$ has only rational singularities 
(in this case since $X$ has only quotient
singularities~\cite[Cor.~4.39]{KM98}, the fact that $X$ has rational singularities follows from the 
definition). Consequently it follows that 

$$
\chi(\wtilde X, \pi^* L)  = \chi (X, L)
$$
and in particular we have $\chi(\sO_{\wtilde X}) = \chi (\sO_X)$.

On the other hand, by Hizerbruch--Riemann--Roch theorem for smooth projective 
threefolds (see~\cite[Ex.~6.7, App.~A]{Har77}) we have

\begin{IEEEeqnarray*}{rCl}\label{eq:HRR3fold}
\chi(\wtilde X, \pi^* L) & = & \frac{1}{12} \cdot  \pi^*L \cdot (\pi^*L -K_{\wtilde X}) \cdot (2\pi^*L - K_{\wtilde X}) \\
\IEEEyesnumber
&&   + \frac{1}{12} \cdot c_2(\wtilde X)\cdot \pi^* L +\chi(\wtilde X,\sO_{\wtilde X}).
\end{IEEEeqnarray*}
Using that fact that $X$ is smooth in codimension two, we now find that the right-hand sides
of (\ref{eq:HRRSing}) and (\ref{eq:HRR3fold}) are equal and therefore Equality (\ref{eq:HRRSing})
is established. 

A similar argument can now be used to show that the equality (\ref{RR2}) also holds.\end{proof}

\end{proposition}

\begin{proposition}[Lower bounds for the Euler characteristic of adjoint bundles]
\label{prop:KeyIneq}
For a terminal projective threefold $X$ the inequality 

\begin{equation}\label{eq:Key}
 \chi(X, K_X+A) \geq  \frac{1}{24} \cdot A \cdot  (K_X+A) \cdot (K_X + 2 A)
\end{equation}
holds, for any Weil divisor $A$ satisfying the following conditions.

\begin{enumerate}
\item $A$ is irreducible. 
\item \label{LCCond} The pair $(X,A)$ is lc and is log-smooth in codimension two.
\item The divisors $A$ and $(K_X+A)$ are Cartier and nef.
\end{enumerate}

\end{proposition}

\begin{proof}
According to (\ref{eq:HRRSing}), with $L$ being replaced by $(K_X+A)$, we have

\begin{IEEEeqnarray*}{rCl}\label{eq:HRR}
\chi(X,  K_X+D+A) & = & \frac{1}{12} \cdot (K_X+A) \cdot A\cdot \bigl( 2(K_X+A) - K_X\bigr)\\
\IEEEyesnumber
&&   + \frac{1}{12} \cdot c_2(X)\cdot (K_X+A) +\chi(X,\sO_X).
\end{IEEEeqnarray*}
Standard Chern class calculations show that we have the equality 

\begin{equation}\label{eq:c_2}
c_2(X) = c_2( \Omega^{[1]}_X\log (A) ) -  (K_X+A)\cdot A - A^2,
\end{equation}
as linear forms on $\N^1(X)_{\bQ}$.
After substituting back into Equality~\ref{eq:HRR}, we find that the equality 

\begin{IEEEeqnarray*}{rCl}\label{eq:long-1}
\chi (X, K_X+D+A) & = & \frac{1}{12}\cdot (K_X+A) \cdot \Big\{A\cdot (K_X+2A)\\ 
&&  + \  \cdot c_2( \Omega^{[1]}_X\log (A) ) -  (K_X+A) \cdot  A  -A^2 \Big\} \\
&& + \ \chi(X,\sO_X)
\end{IEEEeqnarray*}
holds, which then simplifies to 

\begin{equation}\label{eq:long-2}
\chi (X, K_X+A) = \frac{1}{12} (K_X+A) \cdot \Big\{ A^2+  c_2( \Omega^{[1]}_X\log (A) )\Big\} + \chi(X,\sO_X).
\end{equation}

On the other hand, as $X$ is terminal, we know, thanks to~\cite[Lems.~2.2 and~2.3]{Kaw86}, 
that 
\begin{equation}\label{eq:chi}
\chi(X,\sO_X) \geq  \frac{-1}{24} K_X\cdot c_2(X).
\end{equation}
Bu using the equality (\ref{eq:c_2}) we can rewrite this inequality as

\begin{IEEEeqnarray*}{rCl}\label{eq:long-3}
24\cdot \chi (X, \sO_X) & \geq &  \bigl( A - (K_X+A) \bigr)\cdot c_2(\Omega^{[1]}_X\log (A))\\
&& + \ K_X \cdot (K_X+A) \cdot A \\
&& \geq  (K_X+A) \cdot \Big\{   K_X\cdot A  - c_2(\Omega^{[1]}_X\log (A))  \Big\},
\end{IEEEeqnarray*}
where for the latter inequality we have used the assumption that $A$ is nef and 
the pseudoeffectivity of 
$c_2$ (Theorem~\ref{thm:PositiveC2}). Now, going back to Equation~\ref{eq:long-2}, we 
get

\begin{equation}\label{eq:long-4}
24 \chi (X, K_X+A)  \geq  (K_X+A)  \Big\{  2 A^2 + c_2(\Omega_X^{[1]} \log (A) ) + K_X\cdot A      \Big\}  .
\end{equation}
 Again, by using Corollary~\ref{cor:pseff-c_2} and the nefness assumptions on $(K_X+A)$, 
 we find that 
 \begin{equation}\label{eq:long-fin}
 24 \cdot \chi(X, K_X+A) \geq  (K_X+A)  \cdot (2A^2)  +  (K_X+A) \cdot A  =  (K_X+A) \cdot A \cdot (K_X + 2A) ,
 \end{equation}
 as required. 
 \end{proof}


\subsection{Proof of Theorem~\ref*{thm:Main}} 
\label{subsect:ENV}
Thanks to Kawamata--Viehweg vanishing~\cite[Thm.~5.2.7]{Mat02}, it suffices to prove that 
$\chi(Y, K_Y+H)\neq 0$.  
For a general choice of $H'\in |H|$, the pair $(Y, H')$ satisfies Assumption~\ref{LCCond}.
Therefore the assumptions of
Proposition~\ref{prop:KeyIneq} are satisfied except for the terminal 
assumption for the singularities. 

Now, let $\pi: X\to  Y$ be a terminalization of $Y$, cf.~\cite[Sect.~6.3]{KM98}.
Set $A:=\pi^*( H')$.
Since $\pi$ is small, the adjoint divisor $(K_X+A)$ is also nef. 
We may now conclude using the strict positivity of the right-hand side of 
the inequality (\ref{eq:Key}) by the following argument. 
According to the basepoint freeness theorem for log-canonical 
threefolds, cf.~\cite{KMM}, the divisor $K_X+A$ is semi-ample. Therefore, 
for sufficiently large integer $m$, 
we can find an 
irreducible surface $S \in | m\cdot(K_X+2A) |$ such that $(A|_S)$ is  
big. On the other hand, the divisor $(K_X+A)|_S$ is nef. It thus follows that  
$(K_X+A)|_S \cdot A|_S >0$, thanks to Kleiman's ampleness criterion (\cite[Thm.~1.4.29]{Laz04-I}).
\qed

\section{A Miyaoka--Yau inequality in higher dimensions}
In \cite{Miyaoka87}, Miyaoka generalized the famous inequality $c_1^2 \leq 3c_2$ 
from surfaces with pseudoeffective canonical divisor to higher dimensional varieties with nef canonical divisor. 
We extend this result to the case of movable canonical divisor.

\label{sect:section7_MY}

\begin{theorem}\label{HD}
In the setting of Corollary~\ref{cor:pseff-c_2}, we have the inequality 

$$c_1^2(\Omega_{(Y,f,D)}^{[1]})\cdot f^*A \leq 3c_2(\Omega_{(Y,f,D)}^{[1]})\cdot f^*A.$$

\end{theorem}

\begin{proof}
Let $\tilde{H} \in \Amp(X)_{\bQ}$, $H:=f^*\tilde{H}$ and  $\sE:=\Omega_{(Y,f,D)}^{[1]}$. Let $c$ any any positive integer.  
Consider the $\bQ$-twisted 
reflexive sheaf $\sE\langle \frac{1}{c} \cdot H\rangle$.
For the choice of polarization $(f^*A,[\sE\langle \frac{1}{c}\cdot H\rangle])$, 
the assumptions of Proposition~\ref{prop:restrict} are
satisfied, for all $c$.

Now, let $S$ be the complete intersection surface 
defined in Proposition~\ref{prop:restrict} (see also Remark~\ref{rem:Qrestrict})
so that the restriction 
$\sE_S\langle \frac{1}{c}\cdot H_S\rangle:=(\sE\langle \frac{1}{c} \cdot H \rangle)|_S$ 
is semipositive with respect to 
$$\beta:=([\sE]+\frac{r}{c} \cdot H)|_S.$$

Let 
    \begin{equation}\label{eq:HNMY}
       0 \neq \sE_S^1\langle \frac{1}{c}\cdot H_S\rangle \subset \ldots 
          \subset \sE_S^s\langle \frac{1}{c}\cdot H_S \rangle
          =\sE_S\langle \frac{1}{c} \cdot H_S\rangle
              \end{equation}
be the $\bQ$-twisted HN-filtration of $\sE_S\langle 
\frac{1}{c} H_S\rangle$.

The same arguments as those in the proof of Proposition~\ref{prop:c2} show that

$$
(2c_2(\sE_S\langle \frac{1}{c}\cdot H_S\rangle)-c_1^2(\sE_S\langle \frac{1}{c}\cdot H_S\rangle)) \geq (\sum \frac{-1}{r_i}c_1^2(\sQ_S^i)),
$$
where $\sQ_S^i\langle \frac{1}{c}\cdot H_S\rangle$ is the torsion free, $\bQ$-twisted quotient sheaf of rank $r_i$ of 
the filtration (\ref{eq:HNMY}).
 
Again, as in the proof of Proposition~\ref{prop:c2}, for each $i$, we define $\alpha_i$ by the equation 

$$r_i\cdot \alpha_i=\frac{c_1(\sQ_S^i\langle \frac{1}{c}\cdot H_S\rangle)\cdot \beta}{\beta^2}.$$
From the definition of $\alpha_i$, it follows that $\sum r_i \cdot \alpha_i=1$ . 
Moreover, we have $\alpha_1 > \dots > \alpha_s \geq 0,$ where
the last inequality is due to the semipositivity of $\sE_S\langle \frac{1}{c}\cdot H_S\rangle$.

We now deduce 



\begin{align*}
(6c_2(\sE_S\langle \frac{1}{c}\cdot H_S\rangle)- 2c_1^2(\sE_S\langle \frac{1}{c}\cdot H_S\rangle)) \geq \;\;\;\;\;\;\;\;\;\; 
\;\;\;\;\;\;\;\;\;\;  \;\;\;\;\;\;\;\;\;\;  \;\;\;\;\;\;\;\;\;\;  \;\;\;\;\;\;\;\;\;\;  \;\;\; \;\;\;\;\\
\left( 3 (\sum_{i>1} \frac{-1}{r_i}c_1^2(\mathcal{G}_i))+ 6c_2(\sE_S^1\langle \frac{1}{c}\cdot H_S\rangle)-3c_1^2(\sE_S^1\langle \frac{1}{c}\cdot H_S\rangle) +c_1^2(\sE_S\langle \frac{1}{c}\cdot H_S\rangle)\right).
\end{align*}

And finally,
\begin{multline}\label{fin}
(6c_2(\sE_S\langle \frac{1}{c}\cdot H_S\rangle) - 2c_1^2(\sE_S\langle \frac{1}{c}\cdot H_S\rangle)) \geq \\
((1-3 \sum_{i>1} r_i\alpha_i^2).c_1^2(\sE_S\langle \frac{1}{c}\cdot H_S\rangle) + 6c_2(\sE_S^1\langle \frac{1}{c}\cdot H_S\rangle)-3c_1^2(\sE_S^1\langle \frac{1}{c}\cdot H_S\rangle)).
\end{multline}

There are three possibilities: $r_1 \geq 3$, $r_1=2$ and $r_1=1$.

If $r_1 \geq 3$, using Bogomolov--Gieseker inequality and the Hodge index theorem, we obtain
\begin{align*}
(6c_2(\sE_S\langle \frac{1}{c}\cdot H_S\rangle) - 2c_1^2(\sE_S\langle \frac{1}{c}\cdot H_S\rangle))& \geq 
((1-3 \sum_{i>1} r_i\alpha_i^2)\cdot c_1^2(\sE_S\langle \frac{1}{c}\cdot H_S\rangle)-3\frac{1}{r_1}c_1^2(\sE_1))   \\
& \geq(1-3 \sum_{i} r_i\alpha_i^2)\cdot c_1^2(\sE_S\langle \frac{1}{c}\cdot H_S\rangle) \\
& \geq (1-3\alpha_1)\cdot c_1^2(\sE_S\langle \frac{1}{c}\cdot H_S\rangle)  \geq 0  .
\end{align*}
since $3\alpha_1 \leq r_1\alpha_1 \leq \sum_{i} r_i\alpha_i=1.$

\

If $r_1 = 2$, we choose $S$ general enough so that ${\sE_S^1}$ injects into $\Omega^1_S(\log (f^{-1} {\lceil D \rceil}_{|S})).$

Using the Bogomolov--Miyaoka--Yau inequality, we have either $\kappa(S, c_1(\sE_S^1)) \leq 0$ or $c_1^2(\sE_S^1) \leq 3c_2(\sE_S^1).$

In the case $\kappa(S, c_1(\sE_S^1)) \leq 0$, since $c_1(\sE_S^1).\beta>0$, we have $c_1^2(\sE_S^1) \leq 0.$

Applying Bogomolov--Gieseker inequality to \ref{fin}:
\begin{align*}
(6c_2(\sE_S\langle \frac{1}{c}\cdot H_S\rangle)- 2c_1^2(\sE_S\langle \frac{1}{c}\cdot H_S\rangle))  \geq 
((1-3 \sum_{i>1} r_i\alpha_i^2)\cdot c_1^2(\sE_S\langle \frac{1}{c}\cdot H_S\rangle)-\frac{3}{2}c_1^2(\sE_S^1\langle \frac{1}{c}\cdot H_S\rangle)) \\
\geq  
(1-3 \sum_{i>1} r_i\alpha_i^2)\cdot c_1^2(\sE_S\langle \frac{1}{c}\cdot H_S\rangle) -\frac{3}{2}c_1^2(\sE_S^1\langle \frac{1}{c}\cdot H_S\rangle)) \\
\geq
(1-3 \alpha_2\sum_{i>1} r_i\alpha_i)\cdot c_1^2(\sE_S\langle \frac{1}{c}\cdot H_S\rangle) -\frac{3}{2}c_1^2(\sE_S^1\langle \frac{1}{c}\cdot H_S\rangle))   \\
=
(1-3 \alpha_2 (1-2\alpha_1))\cdot c_1^2(\sE_S\langle \frac{1}{c}\cdot H_S\rangle) -\frac{3}{2}c_1^2(\sE_S^1\langle \frac{1}{c}\cdot H_S\rangle))  \\
\geq
(1-3 \alpha_1(1-2\alpha_1))\cdot c_1^2(\sE_S\langle \frac{1}{c}\cdot H_S\rangle) -\frac{3}{2}c_1^2(\sE_S^1\langle \frac{1}{c}\cdot H_S\rangle)) \\
=
\left(6(\alpha_1-\frac{1}{4})^2+\frac{5}{8}\right)\cdot c_1^2(\sE_S\langle \frac{1}{c}\cdot H_S\rangle) -\frac{3}{2}c_1^2(\sE_S^1\langle \frac{1}{c}\cdot H_S\rangle))\geq -\frac{3}{2}c_1^2(\sE_S^1\langle \frac{1}{c}\cdot H_S\rangle)).
\end{align*}

Finally, we obtain $(3c_2(\sE_S)- c_1^2(\sE_S)) \geq 0.$

In the case $c_1^2(\sE_S^1) \leq 3c_2(\sE_S^1)$ we have from \ref{fin}:
\begin{align*}
(6c_2(\sE_S\langle \frac{1}{c} H_S\rangle)- 2c_1^2(\sE_S\langle \frac{1}{c} H_S\rangle)) \geq 
((1-3 \sum_{i>1} r_i\alpha_i^2)c_1^2(\sE_S\langle \frac{1}{c} H_S\rangle)-c_1^2(\sE_S^1\langle \frac{1}{c} H_S\rangle) \\+(6c_2(\sE_S\langle \frac{1}{c} H_S\rangle)-2c_1^2(\sE_S^1\langle \frac{1}{c} H_S\rangle) \\
\geq
((1-4\alpha_1^2-3 \sum_{i>1} r_i\alpha_i^2) c_1^2(\sE_S\langle \frac{1}{c} H_S\rangle))+(6c_2(\sE_S\langle \frac{1}{c} H_S\rangle)-2c_1^2(\sE_S^1\langle \frac{1}{c} H_S\rangle)\\
\geq
((1-4\alpha_1^2-3 \alpha_2\sum_{i>1} r_i\alpha_i) c_1^2(\sE_S\langle \frac{1}{c}\dot H_S\rangle))+(6c_2(\sE_S\langle \frac{1}{c} H_S\rangle)-2c_1^2(\sE_S^1\langle \frac{1}{c} H_S\rangle)\\
=
((1-4\alpha_1^2-3\alpha_2(1-2\alpha_1)) c_1^2(\sE_S\langle \frac{1}{c} H_S\rangle)+(6c_2(\sE_S\langle \frac{1}{c} H_S\rangle)-2c_1^2(\sE_S^1\langle \frac{1}{c}\cdot H_S\rangle))\\
=
(1-2\alpha_1)(1+2\alpha_1-3\alpha_2)\cdot c_1^2(\sE_S\langle \frac{1}{c} H_S\rangle))+(6c_2(\sE_S\langle \frac{1}{c} H_S\rangle)-2c_1^2(\sE_S^1\langle \frac{1}{c} H_S\rangle).
\end{align*}

As $3\alpha_2 < r_1\alpha_1+r_2\alpha_2 \leq 1$, we have
\begin{equation*}
(6c_2(\sE_S)- 2c_1^2(\sE_S))\geq 0.
\end{equation*}

Finally, if $r_1=1$, a classical result of Bogomolov and Sommese (the Bogomolov--Sommese vanishing) implies that 
${\sE_S^1} \subset \Omega^1_S (\log (f^{-1} {\lceil \Delta \rceil}_{|S}))$ has Kodaira dimension at most one. 
Therefore $c_1^2(\sE_S^1) \leq 0.$ From \ref{fin}, one obtains:
\begin{align*}
(6c_2(\sE_S\langle \frac{1}{c}\cdot H_S\rangle)- 2c_1^2(\sE_S\langle \frac{1}{c}\cdot H_S\rangle)) \geq
((1-3 \sum_{i>1} r_i\alpha_i^2)\cdot c_1^2(\sE_S\langle \frac{1}{c}\cdot H_S\rangle)) -3c_1^2(\sE_S^1\langle \frac{1}{c}\cdot H_S\rangle) \\
\geq
((1-3 \alpha_1 \sum_{i>1} r_i\alpha_i)\cdot c_1^2(\sE_S\langle \frac{1}{c}\cdot H_S\rangle)) -3c_1^2(\sE_S^1\langle \frac{1}{c}\cdot H_S\rangle) \\
=
((1-3\alpha_1(1-\alpha_1))\cdot c_1^2(\sE_S\langle \frac{1}{c}\cdot H_S\rangle)) -3c_1^2(\sE_S^1\langle \frac{1}{c}\cdot H_S\rangle) \\
\geq
\left(1-\frac{3}{2}(1-\frac{1}{2})\right)\cdot c_1^2(\sE_S\langle \frac{1}{c}\cdot H_S\rangle)-3c_1^2(\sE_S^1\langle \frac{1}{c}\cdot H_S\rangle) \\
=
\frac{1}{4}c_1^2(\sE_S\langle \frac{1}{c}\cdot H_S\rangle)-3c_1^2(\sE_S^1\langle \frac{1}{c}\cdot H_S\rangle) \\
\geq 
-3c_1^2(\sE_S^1\langle \frac{1}{c}\cdot H_S\rangle).
\end{align*}

Therefore, we have 
\begin{equation*}
(6c_2(\sE_S)- 2c_1^2(\sE_S))\geq 0.
\end{equation*} \end{proof}

We finish this section by pointing out that 
when $(K_X+D)$ is nef, 
the original result of Miyaoka can be adapted to the case of 
orbifold Chern classes. This can then be combined with the semipositivity result of~\cite{CP14}
to conclude the following result. 

\begin{theorem}\label{thm:MYMain}
Let $(X,D)$ be an $n$-dimensional lc pair that is smooth in codimension two.
If $K_X+D$ is nef, then for any strongly adapted
morphism $f:Y\to X$ and any ample divisor $A$ in $X$, we have
\begin{equation}
 c_1^2(\Omega_{(Y,f,D)}^{[1]})\cdot f^*A^{n-2} \leq 
         3c_2(\Omega_{(Y,f,D)}^{[1]})\cdot  f^* A^{n-2}.
\end{equation}
\end{theorem}

\section{Remarks on Lang--Vojta's conjecture in codimension one}
\label{sect:LV}
A classical conjecture of Lang predicts that a variety of general type $X$, 
admits a proper algebraic subvariety that contains all subvarieties of $X$
that are \emph{not} of general type.
In this section, we will prove a particular case of this conjecture for codimension one 
subvarieties satisfying certain conditions: 
the codimension one subvariety will be assumed to be movable and with only canonical singularities.

First, an immediate application of the inequality (\ref{HD}) gives the following theorem.

\begin{theorem}
Let $X$ be a normal, projective and $\bQ$-factorial threefold such that $K_X\in \overline\Mov^1(X)$. 
Let $H$ be a nef divisor, 
$D$ a reduced, irreducible, normal divisor such that $(X,D)$ 
has only isolated lc singularities. 
Assume that $[D] \in \overline\Mov^1(X)$.
If $-K_D$ is pseudoeffective, then
\begin{equation}\label{candeg}
 K_X\cdot D \cdot H \leq (3c_2-c_1^2)\cdot H,
\end{equation}
where $c_i(X) : = c_i( \Omega^{[1]}_X )$.
\end{theorem}

\begin{proof}
As $(K_X+D)\in \overline \Mov^1(X)$, from the inequality (\ref{HD}), we have 
$c_1^2(\Omega^{[1]}_X(\log D)) \cdot H \leq 3c_2(\Omega^{[1]}_X(\log D)) \cdot H.$ 
Therefore, we have

$$
(K_X+D)^2\cdot H \leq 3(c_2+(K_X+D)\cdot D)\cdot H.
$$ 
It follows that 
$$
2K_X\cdot D\cdot H \leq (3c_2-c_1^2)\cdot H+3(K_X+D)\cdot D\cdot H-D^2\cdot H.
$$
Finally, thanks to the adjunction formula, cf.~\cite[Prop.~16.4]{Kol92}, we get $$K_X\cdot D \cdot H \leq (3c_2-c_1^2)\cdot H +2K_D\cdot H|_D.$$
The inequality (\ref{candeg}) now follows from the assumption that $-K_D$ is pseudoeffective.
\end{proof}

\noindent
\emph{Proof of Theorem~\ref{thm:LV1}.}
Let $H$ be an ample divisor in $X$. The divisor $K_X$ is big 
so we can find a positive integer $m$ such that
$(m\cdot K_X - H)$ is linearly equivalent to an effective divisor $E$.

Let us first prove that the family of polarized varieties $(D,H|_D)$ is bounded. 
We note that as each $D$ has only rational singularities, using (\ref{RR2}), we see that
the coefficients of the Hilbert polynomial corresponding to $H|_D$ are determined by Riemann--Roch 
formula.
Therefore, the theorem of Koll\'ar and Matsusaka \cite{KM} 
applies, that is to bound the family $(D, H|_D)$, it suffices to
bound the intersection numbers 

$$
H^2 \cdot D \; \; \; \; \; \;   \;  \text{and} \; \;  \; \; \; \; H \cdot K_D=H\cdot (K_X+D)\cdot D.
$$

For $H^2\cdot D$, we note that, as long as $D$ is not a component of $E$
we can use
the inequality (\ref{candeg}), to get 
$$
0 \leq H^2\cdot D \leq mH\cdot (3c_2-c_1^2). 
$$

For the second term $K_D\cdot H$, 
we use Theorem~\ref{thm:MYIntro} to find 
\begin{align*}
0  & \leq 3c_2(\Omega^{[1]}_X(\log D)) \cdot H-c_1^2(\Omega^{[1]}_X(\log D)) \cdot H\\
  & =  (3c_2-c_1^2)\cdot H+ 2(K_X+D)\cdot D\cdot H-  K_X\cdot D\cdot H.
\end{align*}
We immediately deduce that 

$$ -\frac{1}{2} (3c_2-c_1^2)\cdot H \leq H\cdot (K_X+D)\cdot D=H \cdot K_D \leq 0.$$
Therefore, the family of polarized varieties $(D,H|_D)$ is bounded. 

It now remains to show the finiteness of the polarized varieties $(D, H|_D)$ with canonical singularities and $-K_D$ pseudoeffective. Let $d$ be the Hilbert polynomial of one of this object $D$ and $Hilb_X^d$ the corresponding Hilbert scheme. All other surfaces with canonical singularities in $Hilb_X^d$ are deformations of $D$. From simultaneous resolution of families of surfaces with canonical singularities \cite{KM98}, one can assume that the deformation is smooth. Then from the deformation invariance of the Kodaira dimension \cite{Siu98},  we obtain that such deformations of $D$ are not of general type.
If $Hilb_X^d$ is not finite then $X$ is covered by a family of varieties which are not of general type. This is impossible by the easy additivity of Kodaira dimensions and the fact that $X$ if of general type. The boundedness above gives that polarized varieties $(D, H|_D)$ with canonical singularities and $-K_D$ pseudoeffective are contained in finitely many such Hilbert schemes. This concludes the proof.
\qed

\

\begin{remark}\label{rem:LMProblems}
In~\cite[Thm.~4]{LuMi}, in the setting where $X$ is non-uniruled and smooth and $D$ is reduced, 
the Miyaoka--Yau inequality~\ref{thm:MYMain} is claimed to be valid. 
As a consequence a stronger version of Theorem~\ref{thm:LV1} is obtained. 
Unfortunately we have been unable to verify the details of the proof of~\cite[Thm.~4]{LuMi}. 
The main point of difficulty is that within the proof of this theorem, in~\cite[Subsect.~3.1]{LuMi}, 
the authors claim that given a smooth projective, threefold $X$ of general type with an 
ample divisor $H$, for sufficiently large $m$, there 
is a general member $S\in | m\cdot H|$ for which the following conditions hold.

\end{remark}

\begin{enumerate}
\item \label{item:Wrong1}The restriction $(\Omega^1_X\log (D))|_S$ is semipositive with respect to $(P_{\sigma}(K_X+D))|_S$, 
where $P_{\sigma}$ is the positive part of the divisorial Zariski decomposition of $K_X+D$, cf.~\cite[Chapt.~III]{Nak04}. 
\item \label{item:Wrong2} The restriction $(P_{\sigma}(K_X+D))|_S$ of the positive part of $K_X+D$ verifies the equality 
$P_{\sigma}(K_X+D)|_S \cdot N((K_X+D)|_S)=0$, where $N({K_X+D}|_S)$ is the negative part of the 
Zariski decomposition of the pseudoeffective divisor $(K_X+D)|_S$.
\end{enumerate}
Although Item~\ref{item:Wrong1} in the conditions above can most likely be 
recovered by~\cite[Thm.~2.1]{CP13} and the arguments in Sections~\ref{sect:section3_restriction} 
and~\ref{sect:section4_semipositivity}
in the current paper, the second condition~\ref{item:Wrong2} is more problematic as
the underlying assumption is that Zariski decomposition is functorial; a condition that in 
general does not hold.

\begin{remark}
Starting with a general type variety $X$ and a divisor $D$ such that $(X,D)$ is 
dlt, thanks to~\cite{BCHM10}, it is certainly possible to establish a Miyaoka--Yau 
inequality using a minimal model of $(X,D)$. More precisely, let 
$\pi: (X,D) \dashrightarrow (X', D')$ be a LMMP map resulting in the log-minimal model 
$(X',D')$. Let $\wtilde \pi: \wtilde X \to X'$ be a desingularization of $\pi$ factoring 
through $\mu: \wtilde X\to X$. Now, as
we pointed out prior to Theorem~\ref{thm:MYMain}, thanks to~\cite{CP14},
one can use the original arguments of Miyaoka,
together with those of Megyesi (and his use of $\mathcal Q$-Chern classes), 
to show that the inequality 
$$
\bigl(  3c_2(\Omega^{[1]}_{X'}\log (D') - (K_X'+D')^2)  \bigr)\cdot H^{n-2}\geq 0
$$
holds for any ample divisor $H\subset X'$. 
Furthermore, we can use known 
results on the behaviour of Chern classes under birational morphisms to 
show that 
\begin{equation}\label{eq:MYineqBlowup}
\bigl( 3c_2(\Omega^1_{\wtilde X}\log (\wtilde D)) - (K_{\wtilde X} + \wtilde D)^2 )   \bigr)\cdot 
     \wtilde \pi^*(H)^{n-2}\geq 0.
\end{equation}
But the inequality (\ref{eq:MYineqBlowup}) is hardly independent of the divisor $D$. In fact 
in the inequality (\ref{eq:MYineqBlowup}) even the polarization $(\pi^*H)$ depends on $D$. 
Therefore, the inequality (\ref{eq:MYineqBlowup}) is far from being useful in the context of 
Lang--Vojta's conjecture.

\end{remark}

\end{document}